\documentclass[a4paper,10pt]{amsart}
\usepackage{csquotes}
\usepackage[english]{babel}
\usepackage{algpseudocode}
\usepackage{algorithm}
\usepackage{enumerate}
\usepackage{makecell}
\usepackage{color}
\usepackage{subcaption}
\usepackage{comment}

\usepackage{fullpage}
\usepackage{xspace}             
\usepackage{graphicx}
\usepackage{amsmath,bm}
\usepackage{verbatim}

\usepackage{mathrsfs}
\usepackage{amsmath,amssymb,amsthm}

\usepackage{hyperref}

\usepackage{enumitem}  

\usepackage[backend=biber,
    style=numeric,
    sorting=ydnt,
    natbib=true,
    url=false, 
    doi=true,
    eprint=false,
    giveninits=true,
    maxbibnames=5]{biblatex} 
\numberwithin{equation}{section} 
\newtheorem{theorem}{Theorem}[section]
\newtheorem{corollary}[subsection]{Corollary}
\newtheorem{lemma}{Lemma}[section]

\theoremstyle{definition}
\newtheorem{definition}{Definition}[section]
\newtheorem{remark}[subsection]{Remark}

\newcommand{\1}{{\mathbf{1}}}
\newcommand{\N}{{\mathbb N}}
\newcommand{\T}{{\mathbb T}}
\newcommand{\R}{{\mathbb R}} 
\newcommand{\C}{{\mathbb C}} 
\newcommand{\Z}{{\mathbb Z}}

\newcommand{\E}{{\mathbb E}}
\newcommand{\Prob}{\mathbb{P}}

\newcommand{\dt}{{\Delta t}}

\newcommand{\ESM}{\texttt{ESM}\xspace}
\newcommand{\ExpSM}{\texttt{ExpSM}\xspace}

\newcommand{\TAM}{\texttt{Tam}\xspace}

\renewcommand{\Re}{{\rm Re}}
\renewcommand{\Im}{{\rm Im}}
\newcommand{\ip}[1]{\langle #1 \rangle}
\newcommand{\ipre}[1]{\ip{#1}_{\mathbb{R}}}

\date{\today}

\title{Strong Convergence of a Splitting Method for the Stochastic complex Ginzburg--Landau equation}

\author{Marvin Jans$^*$}

\address{Centre for Mathematical Sciences, Lund University, Lund, Sweden}   
    \email{marvin.jans@math.lth.se, ($^*$Corresponding author)}    
	\author{ Gabriel J. Lord}\address{Mathematics, IMAPP, Radboud University, Nijmegen, The Netherlands
	}
    \email{gabriel.lord@ru.nl}
	\author{ Mariya Ptashnyk}
    \address{Department of Mathematics, Maxwell Institute for Mathematical Sciences, Heriot-Watt University, Edinburgh, Scotland, UK
	}
    \email{m.ptashnyk@hw.ac.uk}



\usepackage{amsopn}

\ifpdf
\hypersetup{
  pdftitle={Strong Convergence of a Splitting Method for the Stochastic complex Ginzburg--Landau equation},
  pdfauthor={M. Jans, G. J. Lord, and M. Ptashnyk}
}
\fi

\begin{document}

\begin{abstract}
We consider the numerical approximation of the stochastic complex Ginzburg-Landau equation with additive noise on the one dimensional torus. The complex nature of the equation means that many of the standard approaches developed for stochastic partial differential equations can not be directly applied.
We use an energy approach to prove an existence and uniqueness result as well to obtain moment bounds on the stochastic PDE before introducing our numerical discretization.
For such a well studied deterministic equation it is perhaps surprising that its numerical approximation in the stochastic setting has not been considered before.
Our method is based on a spectral discretization in space and a Lie-Trotter splitting method in time.
%, similar to that of \cite{brehier_analysis_2019} for the stochastic Allen-Cahn equation.
We obtain moment bounds for the numerical method before proving our main result: strong convergence on a set of arbitrarily large probability. From this we obtain a result on convergence in probability. We conclude with some numerical experiments that illustrate the effectiveness of our method.
\end{abstract}

\maketitle

\section{Introduction}

We consider the numerical approximation of the stochastic complex Ginzburg-Landau equation (SCGLE) with additive space-time noise with periodic boundary conditions, 
\begin{equation}\label{eq:CGLE}
du=\left[(1+i\nu )\Delta u+Ru-(1+i\mu)|u|^2u)\right]dt+\sigma dW.% \qquad  u(0)=u_0\in H^1(\mathbb{T}).
\end{equation}
 The complex Ginzburg-Landau equation arises in various fields in physics and biology, such as evolutionary game theory \cite{frey_evolutionary_2010}, fluid mechanics \cite{GarcaMorales2012} and condensed matter physics \cite{aranson_world_2002}. Without an application in mind, the deterministic complex Ginzburg-Landau equation is an often studied PDE \cite{doering_low-dimensional_1988, Kwasniok2001LowDimensionalMO} in part as there are a number of explicit solutions, such as periodic travelling waves,
 \cite{GarcaMorales2012,SHRAIMAN1992241,aranson_world_2002,doering_low-dimensional_1988}.
 It is common to take the parameter $R>0$ and we assume this throughout the paper.  

Numerical methods for SPDEs with non-globally Lipschitz non-linearities is an active research field and the archetype example is the stochastic Allen-Cahn equation, see for example \cite{Becker2022,brehier_analysis_2019,wang_efficient_2020,BECKER201928,Djurdjevac}.
The stochastic Ginzburg-Landau equation, although similar to the stochastic Allen-Cahn equation, poses a number of particular technical issues. A notable difference is that the SCGLE \eqref{eq:CGLE} is complex-valued and due to the complex coefficient in front of the Laplacian, the generated semigroup is not a contracting semigroup in the space of continuous functions. As a consequence the standard techniques for non-linear locally Lipschitz SPDEs, as described in  \cite{da_prato_zabczyk_2014,Cerrai2001SecondOP}, do not apply. There are only a limited number of results for \eqref{eq:CGLE}, see for example \cite{Blomker,CHENG202358,guo_attractor_2008,kuksin_randomly_2004,odasso_ergodicity_2006,trenberth_global_2019,weinan_renormalized_2016} and we are not aware of any results on strong convergence of a numerical method. The work of \cite{Blomker} considers attractors for the SCGLE on unbounded domains, \cite{guo_attractor_2008} also examines attractors whereas \cite{CHENG202358} examines the averaging principle using a variational approach. The work of \cite{kuksin_randomly_2004} examines stationary measures and the inviscid limit and \cite{odasso_ergodicity_2006} looks at ergodicity.
The SCGLE on the two dimensional torus with space-time white noise is considered in \cite{trenberth_global_2019} and a global existence result is proved using renormalization, see also \cite{weinan_renormalized_2016}.

We construct our numerical method by using a spectral Galerkin method in space and a Lie-Trotter splitting method in time where the complex nonlinear ODEs are solved exactly and are coupled to the complex stochastic convolution.
This is similar to that proposed in \cite{brehier_analysis_2019} for the Allen-Cahn equation and \cite{Goldman1995ANM} for the deterministic case.
In this paper, we present the proofs for the moment bounds for the periodic boundary condition and relevant norms needed for the analysis of our numerical method. 
These differ from the moment bounds in \cite{Blomker,CHENG202358}, obtained   for different norms and in a different setting. 

The paper is organised as follows, first, we present some standard results for fractional Sobolev spaces, semigroups, and stochastic processes and give some slight variations of these results which are specific to our problem. In Section \ref{sec_SCGLE}, we show that the SCGLE \eqref{eq:CGLE} is well-defined and prove moment bounds for the solution, which are needed to analyse the numerical methods. For this we need to impose constraints on the regularity of the noise
and on the regularity of the initial data.
In Section \ref{sec_Num} the numerical method is introduced.
To obtain moment bounds on the scheme we need to impose a further constraint on the parameter $\nu$ so that $|\nu|\leq \sqrt{3}$. We then prove the strong convergence for the numerical method and as a corollary convergence in probability. In Section \ref{sec_exp}, we discuss the implementation and perform numerical tests to verify our theoretical results and compare with results  when regularity assumptions and conditions on the model parameters are relaxed. 

\section{Mathematical setting and stochastic convolution} \label{Back}

We consider \eqref{eq:CGLE} on the domain $[0,1]$ with periodic boundary conditions, 
i.e.~on an $1$-dimensional torus~$\mathbb{T}$.
For the Hilbert space $L^2(\T;\C)$ we use the notation $L^2$ and denote the norm  by $\|\cdot \|:=\|\cdot \|_{L^2}$, the inner-product by $\langle\cdot,\cdot\rangle$, and   the real part of the inner-product $\Re\{\ip{\cdot,\cdot}\}$ by $\ipre{\cdot,\cdot}$. 
We consider the basis for  $L^2$  given by the eigenfunctions $\{\phi_k\}$ of $-\Delta$, with the corresponding eigenvalues $\lambda_k$, 
 \begin{equation}\label{eq:phi}
    \phi_k(x):=e^{i2\pi kx}\in C^{\infty}(\T;\C), \qquad \lambda_k:=(2\pi k)^2, \qquad \text{ for } \; k\in\mathbb{Z}. 
\end{equation}
We denote the Banach space of complex-valued continuous functions on $\T$ with the usual supremum norm by $V:=C(\T,\C)$.
For separable Hilbert spaces $U$ and $H$ with norms $\|\ \|_U$ and $\|\ \|_H$ and a  linear operator 
$D: U \to H$, the operator norm  and  the Hilbert-Schmidt norm  are defined as
$$\|D\|_{\mathcal{L}(U;H)}=\sup_{u\in U,\|u\|_U=1}\|Du\|_H, \qquad \|D\|_{\mathcal{L}^2(U;H)}=\Big(\sum_{k=1}^{\infty}\|D\chi_k\|_{H}^2\Big)^{\frac{1}{2}},$$
where $\{\chi_k\}_{k\in\N}$ is an orthonormal basis for $U$.  When $U=H$, we write $\|D\|_{\mathcal{L}(U)}$ and $\|D\|_{\mathcal{L}^2(U)}$. 
It is straightforward to show that if $D$ is a Hilbert-Schmidt operator, it is also a bounded linear operator. Recall the following inequality for two linear operators, $D$ and $E$, on the separable Hilbert space $H$,
\begin{equation} \label{eq:Hilbertt_op_ineq}
\|DE\|_{\mathcal{L}^2(H)}^2=\sum_{k=1}^{\infty}\|DE\chi_k\|^2_{H}\leq\sum_{k=1}^{\infty}\|D\|_{\mathcal{L}(H)}^2\|E\chi_k\|_{H}^2=\|D\|^2_{\mathcal{L}(H)}\|E\|_{\mathcal{L}^2(H)}^2.
\end{equation}

The fractional Sobolev spaces are defined for $\alpha\in\R$ by
$$
\dot H^\alpha (\T):= \mathcal D(\tilde\Delta^{\alpha/2}), \quad \text{ with } \; \;  \|u\|^2_{\dot H^\alpha(\T)} := \|u\|^2_{\dot H^\alpha} = \langle \tilde{\Delta}^{\frac{\alpha}{2}}u,\tilde{\Delta}^{\frac{\alpha}{2}}u\rangle, 
$$
where $\tilde \Delta := I-\Delta$ with periodic boundary conditions. 

For the SCGLE we are interested in semigroups $e^{tA}$ that have as a generator 
\begin{equation}
\label{eq:A defn}
    A:=(1+i\nu)\Delta,
\end{equation}
with periodic boundary conditions. 
For the definition and properties of semigroups and fractional Sobolev spaces see, for example,  ~\cite{2000, Henry_1981, Pazy_1983, thomee}.
The operator $A$ is a  linear closed  operator on $\dot{H}^{\alpha}$ with domain of definition  $\dot{H}^{\alpha+2}$ and the spectrum is given by  $\{-(1+i\nu)\lambda_k\}_{k\in\Z}$, where the $\lambda_k$ are as in \eqref{eq:phi},
and for $\lambda>0$ we have
$$\|(\lambda I-A)^{-1}\|_{\mathcal{L}(\dot{H}^{\alpha})}\leq\frac{1}{\lambda}.$$
Hence $-A$ is sectorial and $A$ is a generator of $C_0$ and analytic semigroup of contraction on $\dot H^\alpha$, see e.g.~\cite{da_prato_zabczyk_2014, 2000, Henry_1981, Pazy_1983}. 

\begin{remark}
    In the estimates below we do not keep track of the dependency of constants $C$ on the model parameters $\mu$, $\nu$, $R$ and initial condition $u_0$.
    For other parameters we indicate by subindex if a constant depends on them, e.g. $C_T$ is a constant that depends on the final time $T$.
\end{remark}
The following lemma collects three useful properties of the semigroup $e^{tA}$.
   \begin{lemma}\label{lem:semigroup}
 (i)   For  $u\in \dot{H}^{\beta}$ and $\beta\leq\alpha$ there is a constant $C_{\alpha,\beta,T}>0$ such that
    \begin{equation*}
        \|e^{tA}u\|_{\dot{H}^{\alpha}}\leq C_{\alpha,\beta,T}t^{\frac{\beta-\alpha}{2}}\|u\|_{\dot{H}^{\beta}} \qquad \text{ for }  t \in (0,T].
    \end{equation*}
(ii) For $\gamma\geq 0$, there exists a constant $C_{T,\gamma}>0$ such that for $t \in (0,T]$
$$\|\tilde{\Delta}^{\gamma}e^{tA}\|_{\mathcal{L}(L^2)}\leq C_{T,\gamma}t^{-\gamma}.
$$
(iii)    For $\gamma \in  [0,1]$ and $t\in (0,T]$, there exists a constant $C_{T,\gamma}>0$ such that
$$\|\tilde{\Delta}^{-\gamma}(I-e^{tA})\|_{\mathcal{L}(L^2)}\leq C_{T, \gamma}t^{\gamma}.$$
\end{lemma}
\begin{proof}
The proofs for general analytic semigroups can be found in e.g.~\cite{2000,  Henry_1981, Pazy_1983}.
We include the main ideas for completeness in this complex setting. 

  $(i)$
  Using the definition of the norm in $\dot H^\alpha$ we obtain   
     $$
     \begin{aligned} 
     \|e^{tA}u\|_{\dot{H}^{\alpha}}
  &  = e^{t} \Big[\sum_{k\in\Z}(1+\lambda_k)^{\alpha-\beta} e^{-2(1+\lambda_k) t}(1+\lambda_k )^\beta|a_k|^2\Big]^{\frac{1}{2}}
     \\
    & \leq e^t\Big[\sum_{k\in\Z }\left(\frac{\alpha-\beta}{2t}\right)^{\alpha-\beta} e^{-(\alpha-\beta)}(1+\lambda_k)^{\beta}|a_k|^2\Big]^{\frac{1}{2}}\leq C_{\alpha,\beta,T} t^{\frac{\beta-\alpha}{2}}\|u\|_{\dot{H}^{\beta}}.
     \end{aligned} 
     $$
     Here we used that  $f(x)=x^{\alpha-\beta} e^{-2t x}$ takes its maximum at $x= (\alpha-\beta)/2t$.
     
$(ii)$ This follows from the relation 
$$\tilde{\Delta}^{\gamma}e^{tA}\phi_k=(1+\lambda_k)^{\gamma}e^{-(1+i\nu)\lambda_kt}\phi_k,$$
and the fact that the function $f(x)=x^{\gamma}e^{-xt}$ takes its maximum at $x=\gamma/ t$, and that  $|e^{-i\nu(\lambda_k+1)t}|=1$. 

Finally to show $(iii)$ we can
use the estimate from $(ii)$ to obtain 
$$
\|\tilde \Delta^{-\gamma}(I - e^{tA}) u\| = \Big\|\int_0^t \tilde \Delta^{-\gamma} A e^{sA} \, u\,  ds\Big\| \leq C_{T,\gamma} \int_0^t s^{\gamma-1} \|u\| ds \leq C_{T,\gamma} t^\gamma \|u\|.
$$
\end{proof}

\subsection{Stochastic convolution}
Since most standard text books, e.g.  \cite{lord_powell_shardlow_2014,da_prato_zabczyk_2014,Cerrai2001SecondOP}, consider a real Wiener processes,  we present here some results for the complex case. Consider  a filtered probability space $(\Omega,\mathscr{F},\mathscr{F}_t,\Prob)$ and a separable Hilbert space $X$.
\begin{definition}\label{qk}
We define an $\mathscr{F}_t$-adapted 
stochastic process $\{W(t)\}_{t\geq 0}$ as complex $\dot{H}^r-$valued noise if for  $\varepsilon\in(0,\frac{3}{2})$  there is a set of positive numbers $\{q_k\}_{k\in\Z}$,  with $q_k\leq C|k|^{-2r-1-2\varepsilon}$, such that $W(t)$ has the same distribution as
\begin{equation}\label{eq:W_rep}
W(t)=\sum_{k\in\Z}\sqrt{q_k}(\beta^r_k(t)+i\beta^i_k(t))\phi_k,
\end{equation}
where $\beta^r_k(t)$ and $\beta^i_k(t)$ are i.i.d.~Brownian motions and $\{\phi_k\}_{k \in \mathbb Z}$ is an orthonormal basis for $X$.
\end{definition}
The representation \eqref{eq:W_rep} can also be written as
\begin{equation}\label{eq:BA_noise}
W(t) = \sum_{k\in\Z}B\tilde{\Delta}^{-\frac{r}{2}-\frac{1}{4}-\frac{\varepsilon}{2}}(\beta^r_k(t)+i\beta^i_k(t))\phi_k,
\end{equation}
where $B\in\mathcal{L}(L^2)$ and
$$B\tilde{\Delta}^{-\frac{r}{2}-\frac{1}{4}-\frac{\varepsilon}{2}}\phi_k=\sqrt{q_k}\phi_k.$$
In our analysis below we exploit this property of $B$.
For  the complex $\dot{H}^r-$valued noise $\{W(t)\}_{t\geq 0}$  the stochastic convolution, given by
\begin{equation}\label{def_stoch_conv}
\begin{aligned} 
\eta(t)& :=\sigma \int_0^t e^{(t-s)A}dW(s)\\
&=\sigma \sum_{k\in\Z}\left[\int_0^t \sqrt{q_k}e^{(t-s)A}d\beta_k^r(s)+i\int_0^t \sqrt{q_k}e^{(t-s)A}d\beta_k^i(s)\right]\phi_{k}\\
&= \sigma \sum_{k\in\Z}\sqrt{q_k}\Big(\int_0^te^{(s-t)(1+i\nu)\lambda_k}d\beta^r_k(s)+i\int_0^te^{(s-t)(1+i\nu)\lambda_k}d\beta^i_k(s)\Big)\phi_k,
\end{aligned} 
\end{equation}
is the mild solution of the SPDE for $t\in (0,\infty)$
\begin{equation*}
       d\eta =A\eta dt+\sigma dW(t), \qquad 
       \eta(0)=0.
\end{equation*}
From now on we also do not include the dependence on the constants $C>0$ on the parameters $\varepsilon$ and $r$ from Definition~\ref{qk}.
\begin{lemma}\label{stoch_app}
For  $p\in[2,\infty)$, $T>0$, there exists a constant $C_{ T,p}>0$ such  that
$$\E\big[\sup_{t\in[0,T]}\|\eta(t)\|^p_{\dot{H}^{r+1}}\big]\leq C_{T,p}.$$
%where $\eta$ is defined as in \eqref{def_stoch_conv}.
\end{lemma}
\begin{proof}
    First consider  $p>6/\varepsilon$. %, where $\varepsilon$ is as in Definition~\ref{qk}. 
    Using the Da Prato-Kwapie\'{n}-Zabczyk factorization method \cite[Theorem~5.10]{DaPrato2004} we  can rewrite the stochastic convolution as
    $$\eta(t)=\sigma\frac{\sin(\pi \gamma)}{\pi}\int_0^t (t-\tau)^{\gamma-1}\! e^{(\tau-t)A}Y(\tau)d\tau, \; \text{ with }
    Y(\tau)=\int_0^{\tau}(\tau-s)^{-\gamma}\! e^{(\tau-s)A}dW(s),
    $$
    for $\gamma\in(0,\frac{1}{2})$.
    Choosing $\gamma =  \frac{\varepsilon}{6}$ and applying  H\"older's inequality yield
    \begin{align*}
        \|\eta(t)\|_{\dot{H}^{r+1}}^p&\leq \sigma\frac{\sin(\pi \frac{\varepsilon}{6})^p}{\pi^p}\Big(\int_0^t (t-\tau)^{(\frac{\varepsilon}{6}-1)\frac{p}{p-1}} d\tau\Big)^{p-1}\int_0^t\|Y(\tau)\|^p_{\dot{H}^{r+1}}d\tau\\
        &\leq C_{T,p}\int_0^T\|Y(\tau)\|^p_{\dot{H}^{r+1}}d\tau, 
    \end{align*}
 for all $t \in [0,T]$ and $p>6/\varepsilon$.    Furthermore, using the Burkholder-Davis-Grundy inequality \cite{Liu2015} we obtain
    \begin{align*}
        \E(\|Y(\tau)\|^p_{\dot{H}^{r+1}})&\leq C_p \Big(\int_0^{\tau} (\tau-s)^{-\frac{\varepsilon}{3}}\|\tilde{\Delta}^{\frac{r+1}{2}}e^{(\tau-s)A}B\tilde{\Delta}^{-\frac{r}{2}-\frac{1}{4}-\frac{\varepsilon}{2}}\|^2_{\mathcal L^2(L^2)}ds\Big)^{\frac{p}{2}}\\
        &\leq C_p \big(\int_0^{\tau} (\tau-s)^{-1+\frac{\varepsilon}{3}}\|\tilde{\Delta}^{-\frac{1}{4}-\frac{\varepsilon}{6}}\|^2_{\mathcal L^2(L^2)}ds\big)^{\frac{p}{2}} \\
        &\leq C_{p,T} \|\tilde{\Delta}^{-\frac{1}{4}-\frac{\varepsilon}{6}}\|^p_{\mathcal L^2(L^2)}\leq  C_{p,T} \Big(\sum_{k\in\Z} k^{-1-\frac{2\varepsilon}{3}}\Big)^{\frac{p}{2}}\leq C_{T,p}, 
    \end{align*}
   and hence the estimate stated in the lemma for $p>6/\varepsilon$.   For  $2\leq p \leq 6/\varepsilon$, we have
   \begin{align*}
        \E\big[\sup_{t\in[0,T]}\|\eta(t)\|^p_{\dot{H}^{r+1}}\big] & \leq \E\big[\mathbf{1}_{\sup_{t\in[0,T]}\|\eta(t)\|^p_{\dot{H}^{r+1}}\leq 1}\big]\\
         &\quad +\E\big[\mathbf{1}_{\sup_{t\in[0,T]}\|\eta(t)\|^p_{\dot{H}^{r+1}}>1}\|\eta(\tau)\|^{\frac{7}{\varepsilon}}_{\dot{H}^{r+1}}\big] \leq 1+\E\big[\|\eta(\tau)\|^{\frac{7}{\varepsilon}}_{\dot{H}^{r+1}}\big].
   \end{align*}
    This yields the result for all $p\in [2, \infty)$. 
\end{proof}

\begin{lemma}\label{stoch_time_reg}
    For the stochastic convolution $\eta$, defined in \eqref{def_stoch_conv}, and $r\geq 0$, there is a $C_T>0$ such that
    $$\E\big[\|\eta(t)-\eta(s)\|^2\big]\leq C_T (t-s) \quad \text{ for all } \; 0 \leq s\leq t \leq T.$$
\end{lemma}
\begin{proof}
 From the definition of $\eta$ it follows that 
 \begin{equation*}
     \begin{split}
        \E\big[\|\eta(t)-\eta(s)\|^2\big]&=\sigma\E\left[\Big\|\int_0^s \!\! \big(e^{(t-\tau)A}-e^{(s-\tau)A}\big)dW(\tau)+\int_s^t \!\! e^{(t-\tau)A}dW(\tau)\Big\|^2\right]\\
         &=\sigma\E\left[\Big\|\int_0^s\!\!\! e^{(s-\tau)A}\big(e^{(t-s)A}-I\big)dW(\tau)\Big\|^2\right]+\sigma\E\left[\Big\|\int_s^t \!\!\! e^{(t-\tau)A}dW(\tau)\Big\|^2\right]=:\sigma(I+II).
     \end{split}
 \end{equation*}
To obtain $I$ and $II$ we used the fact that the  stochastic integrals over $(0,s)$ and $(s,t)$ are independent and their expected value is zero. 

  Let us examine $I$. Using the It\^o isometry, see \cite[Theorem 10.16]{lord_powell_shardlow_2014}, $\eta(t)$ has the same law as~\eqref{eq:BA_noise}, and the fact that $W$ is given by two independent Wiener processes, one for the real and one for the imaginary part, we obtain 
 \begin{equation*}
  \begin{split}
  I& = 2\int_0^s\big\|e^{(s-\tau)A}(e^{(t-s)A}-I)B\tilde{\Delta}^{-\frac{r}{2}-\frac{1}{4}-\frac{\varepsilon}{2}}\big\|^2_{\mathcal{L}^2(L^2)}d\tau\\
   &\leq 2\int_0^s\|\tilde{\Delta}^{\frac{1}{2}-\frac\varepsilon 8}e^{(s-\tau)A}\tilde{\Delta}^{-\frac{1}{2}}(e^{(t-s)A}-I) \tilde{\Delta}^{\frac\varepsilon 8} B\tilde{\Delta}^{-\frac{r}{2}-\frac{1}{4}-\frac{\varepsilon}{2}}\|^2_{\mathcal{L}^2(L^2)}d\tau .
   \end{split}
   \end{equation*}
   Applying inequality~\eqref{eq:Hilbertt_op_ineq} and using that  the operators are diagonal with respect to the Fourier basis
\begin{equation*}
  \begin{split}
  I& \leq C\int_0^s\|\tilde{\Delta}^{\frac{1}{2}-\frac\varepsilon 8}e^{(s-\tau)A}\|^2_{\mathcal{L}(L^2)}\big\|\tilde\Delta^{-\frac{1}{2}}(e^{(t-s)A}-I)\big\|^2_{\mathcal{L}(L^2)}\|\tilde{\Delta}^{\frac\varepsilon 8} B\tilde \Delta^{-\frac{r}{2}-\frac{1}{4}-\frac{\varepsilon}{2}}\|^2_{\mathcal{L}^2(L^2)}d\tau.  
  \end{split}   
 \end{equation*}
 Using  Lemma~\ref{lem:semigroup}
  and  that $\|\tilde{\Delta}^{\frac\varepsilon 8} B\tilde \Delta^{-\frac{r}{2}-\frac{1}{4}-\frac{\varepsilon}{2}}\|_{\mathcal{L}^2(L^2)}$ is finite for $r\geq 0$, we have
 \begin{equation*}
 \centering
 \begin{split}
   I&\leq C_T\int_0^s \frac{t-s}{(s-\tau)^{1-\varepsilon/4}}d\tau\leq C_T(t-s).
 \end{split}   
 \end{equation*}
 For the second term $II$, we use Itô's isometry and Lemma~\ref{lem:semigroup} to obtain
 \begin{equation*}
     \begin{split}
       II&\leq  \int_s^t \big\|e^{(t-\tau)A} B\tilde{\Delta}^{-\frac{r}{2}-\frac{1}{4}-\frac{\varepsilon}{2}}\big\|^2_{\mathcal{L}^2(L^2)}d\tau\leq \int_s^t \|e^{(t-\tau)A}\|^2_{\mathcal{L}(L^2)}\|B\tilde{\Delta}^{-\frac{r}{2}-\frac{1}{4}-\frac{\varepsilon}{2}}\|^2_{\mathcal{L}^2(L^2)}d\tau \leq C_T(t-s).
     \end{split}
 \end{equation*}
 Combining the estimates on $I$ and $II$ implies the result stated in lemma.
\end{proof}

\section{Existence, uniqueness, and moment bounds of the SCGLE}
\label{sec_SCGLE}
To prove convergence of our numerical method we need moment bounds for the mild solution of~\eqref{eq:CGLE} and of an  auxiliary SPDE, defined later in the text. In this section we proof moment bounds for the mild solution of~\eqref{eq:CGLE}, but similar arguments can be applied for the mild solution of the auxiliary SPDE. 
Although existence of a weak solution of \eqref{eq:CGLE} was shown in \cite{CHENG202358} we require higher regularity and moment bounds in $H^1$ not shown in \cite{CHENG202358}. Furhtermore, we have no restriction on $\mu$ unlike in \cite{CHENG202358}.

\subsection{The non-linear term}
Before deriving a priori estimates for the SCGLE \eqref{eq:CGLE} and proving the well-posedness result, we examine properties of the term
\begin{equation}
    \label{eq:Psi_0}
\Psi_{0}(u):=Ru-(1+i\mu)|u|^2u.
\end{equation}
For  $u,v\in V$, using  the triangle and Young's inequalities,
and  considering the $L^2$ and $V$ norms, respectively,  we obtain
\begin{equation}\label{eq:loc_lip_L2}
\begin{aligned}
   \|\Psi_{0}(u)-\Psi_{0}(v)\| &\leq R\|u-v\| + \frac{3}{2}\sqrt{1+\mu^2}\big(\|u\|^2_V+\|v\|^2_V\big)\|u-v\|,\\
    \|\Psi_{0}(u)-\Psi_{0}(v)\|_V &\leq R \|u-v\|_V +\frac{3}{2}\sqrt{1+\mu^2}\big(\|u\|^2_V+\|v\|^2_V \big)\|u-v\|_V.
   \end{aligned} 
\end{equation}
Choosing  $v=0$ in \eqref{eq:loc_lip_L2} gives a bound for $\Psi_{0}(u)$  with respect to the $L^2$-norm
\begin{equation}\label{eq:F_bound}
   \|\Psi_0(u)\|\leq R\|u\|+\frac{3}{2}\sqrt{1+\mu^2}\|u\|_V^2\|u\|\leq C(1+\|u\|_V^2)\|u\|, \quad \text{ for } \; u \in V.
\end{equation}
We also use the following inequality. 
\begin{lemma}\label{lipy}
    There is a constant $C>0$ such that 
    $$\big\langle\Psi_0(u+v), u\big\rangle_{\R}\leq C\big(1+|u|^2+|v|^4\big).$$
\end{lemma}
\begin{proof}
We divide the proof into two cases. 
First consider $|v|>\frac{1}{1+|\mu|}|u|$.
Using that $\Psi_0(\cdot)$ can be bounded by a cubic function, yields 
$$\langle\Psi_0(u+v), u\rangle_{\R}\leq |\Psi_0(u+v)||u|\leq C (1+|u|^3+|v|^3)|v|\leq C(1+|v|^4).
$$
In the second case, when
$|v|\leq\frac{1}{1+|\mu|}|u|$,
writing $v=v_1+i v_2$ for  $v_1,v_2\in\R$, we obtain 
$$
\begin{aligned}
\langle(1+i\mu)|u+v|^2(u+v),  u\rangle_{\R}&= |u+v|^2(|u|^2 +  u_1(v_1-\mu v_2) + u_2(v_2 + \mu v_1))\\
&\geq |u+v|^2\big(|u|^2 -  (1+ |\mu|) |u||v|\big)  \geq 0.
\end{aligned}
$$
Hence 
$$
\langle\Psi_0(u+v), u\rangle_{\R}=R\langle u+v, u\rangle_{\R}-\langle(1+i\mu)|u+v|^2(u+v), u\rangle_{\R} \leq
C(|u|^2+|v|^2).
$$
\end{proof}

\subsection{A priori estimates}
The mild solution for the SCGLE~\eqref{eq:CGLE} is given by 
\begin{equation}\label{eq:mild_SCGLE}
u(t)=e^{A t}u_0+\int_0^t e^{(t-s)A}\Big( Ru(s)-(1+i\mu)|u(s)|^2u(s)\Big) ds+ \sigma \int_0^te^{(t-s)A} dW(s). %}_{=\eta(t)}.  
\end{equation}
To show the existence almost surely (a.s) of  a unique mild solution of the SCGLE~\eqref{eq:CGLE}, 
we consider $v(t)=u(t)-\eta(t)$ pathwise, where $v$ satisfies
\begin{equation}\label{eq:RPDE}
\begin{aligned}
\partial_t v &= (1+i\nu)\Delta v(t)+R(v(t)+\eta(t))-(1+i\mu)|v(t)+\eta(t)|^2(v(t)+\eta(t)), \\
v(0)& =u_0,
\end{aligned}
\end{equation}
and $\eta$ is the stochastic convolution defined in \eqref{def_stoch_conv}.
For the noise $W(t)$ which is $\dot{H}^r$-valued with $r\geq0$,  the stochastic convolution $\eta \in L^p(\Omega;L^{\infty}(0,T;\dot{H}^{\alpha}))$ for all $\alpha\leq 1+r$ and  $p\geq1$. 
Our approach to prove the   pathwise existence and uniqueness results for~\eqref{eq:CGLE} is to use the Galerkin method and energy estimates. We first prove a priori estimates, assuming existence of sufficiently regular solutions. Then we apply these estimates to the Galerkin approximation, and in the limit we obtain existence of the mild solution for~\eqref{eq:CGLE}. 
\begin{lemma}\label{energy}
    Let $v \in L^2(0,T;\dot{H}^1)$ be a weak solution of~\eqref{eq:RPDE} and let $u_0\in L^2$. 
    Then there exists a constant $C_T>0$ such that
    \begin{equation}\label{estim_H1}
    \sup_{t\in[0,T]}\|v(t)\|^2+\| v\|^2_{L^2(0,T;\dot{H}^1)}\leq C_T\big(1+\|u_0\|^2+\|\eta\|_{L^4((0,T)\times \mathbb T)}^4\big).
    \end{equation}
   If  $v\in L^2(0,T;\dot{H}^2(\mathbb T))$ and $u_0\in \dot{H}^1$,  then %there exists a constant $C_T>0$ such that
    \begin{equation}\label{estim_H2}
   \sup\limits_{t\in [0,T]} \|v(t)\|_{\dot{H}^1}^2 +\|v\|^2_{L^2(0,T; \dot{H}^2)}\leq C_T\big(1+ \|u_0\|^2_{\dot{H}^1}+\sup_{t\in[0,T]}\|\eta(t)\|_{L^2}^{12} + \|\eta\|^{12}_{L^2(0,T; \dot H^1)}\big).
    \end{equation}
If additionally   $\partial_t v\in L^2((0,T)\times\mathbb T)$, then
    \begin{equation}\label{estim_part_t}
    \|\partial_t v\|^2_{L^2((0,T)\times \mathbb T)}\leq C_T\big(1+\|u_0\|^2_{\dot{H}^1}+\sup_{t\in[0,T]}\|\eta(t)\|_{L^2}^{12} + \|\eta\|^{12}_{L^2(0,T; \dot H^1)}\big).
    \end{equation}
\end{lemma}
\begin{proof}
Taking $\overline{v}(t)$ as test function for the weak solution of \eqref{eq:RPDE}, integrating over $\T$,  taking the real part  and using Lemma~\ref{lipy}  yields
$$\frac{1}{2}\frac{d}{dt}\|v(t)\|^2+\|\nabla v(t)\|^2\leq C\big(1+\|v(t)\|^2+\|\eta(t)\|^4_{L^4}\big).$$
Integrating with respect to time and using Grönwall's inequality implies
$$
\|v(t)\|^2+\int_0^t \|\nabla v(s)\|ds \leq  C_T\Big[\|u_0\|^2 + \int_0^t[1+\|\eta(s)\|^4_{L^4}]ds\Big],
$$
for $t \in (0,T]$. Taking the supremum over $[0,T]$ yields the desired result.

Multiplying \eqref{eq:RPDE} by $-\Delta\overline{v}(t)$, integrating over $\T$,  taking the real part,
and using Hölder's inequality implies
$$\frac{1}{2}\frac{d}{dt}\|\nabla v(t)\|^2+\|\Delta v(t)\|^2\leq C\| \Psi_0(v(t)+\eta(t))\|^2+\frac{1}{2}\|\Delta v(t)\|^2.$$
Since $|\Psi_0(z)|\leq C(1+|z|^3)$, we get
$$\frac{1}{2}\frac{d}{dt}\|\nabla v(t)\|^2+\frac{1}{2}\|\Delta v(t)\|^2\leq C(1+\| v(t)+\eta(t)\|_{L^6}^6).$$
Now we use the Gagliardo-Nirenberg interpolation inequality for $u\in \dot{H}^1$, see e.g.~\cite{Brezis2011},
$\|u\|_{L^6}\leq C\| u\|_{\dot{H}^1}^{\frac{1}{3}}\|u\|^{\frac{2}{3}}$,
to obtain 
$$\frac{d}{dt}\|\nabla v(t)\|^2+\|\Delta v(t)\|^2\leq C\big(1+\| v(t)+\eta(t)\|_{\dot{H}^1}^{2}\|v(t)+\eta(t)\|^{4}\big).$$
Integrating over $t$,
taking the supremum over $[0,T]$ of $(\|v(s)\|^{4}+\|\eta(s)\|^{4})$,  and applying the inequality of~\eqref{estim_H1}  yields
$$
\begin{aligned}
&\|\nabla v(t)\|^2+\int_{0}^t \|\Delta v(s)\|^2ds\leq |v(0)|_{\dot{H}^1}^2 \\
& + C_T \big(1+\sup_{s\in[0,T]}\|\eta(s)\|^{4} + \|\eta\|^{8}_{L^4((0,T)\times \mathbb T)}\big)\big(1+ \|\eta\|_{L^4((0,T)\times \mathbb T)}^4 + \|\eta\|_{L^2(0,T;\dot{H}^1)}^{2}\big).
\end{aligned}
$$
Then by the Gagliardo-Nirenberg inequality we have
$$
 \|\eta\|^{8}_{L^4((0,T)\times \mathbb T)}
 \leq C_T\int_0^T \|\eta(t)\|^6 dt \int_0^T \|\eta(t)\|^2_{\dot H^1} dt,
$$
and hence  \eqref{estim_H2}. To prove~\eqref{estim_part_t}, multiply~\eqref{eq:RPDE} with $\partial_t\overline{v}(t)$ and integrate over $\T\times [0,t]$
$$
\begin{aligned} 
\int_0^t\|\partial_sv(s)\|^2ds
\leq 2(1+ |\nu|) \int_0^t\big[\|\Delta v(s)\|^2+\|\Psi_0(v(s)+\eta(s))\|^2 \big]ds.
\end{aligned}
$$
The first term on the right hand side  can be bounded using~\eqref{estim_H2}, whereas the second term is bounded in the same way as in the derivation of the estimate~\eqref{estim_H2}. 
\end{proof}

\subsection{Existence and uniqueness results}

We show existence of a solution to the SCGLE~\eqref{eq:CGLE}, using the Galerkin method. We define the projection $P_N$ of $u \in \dot H^\beta$, for $\beta \in \mathbb R$, on the $N$-dimensional subspace 
 \begin{equation}\label{eq:Proj_def}P_Nu:=\sum_{k=k_0}^{\lfloor\frac{N}{2}\rfloor}a_k\phi_k,\end{equation}
where $k_0=-\lfloor\frac{N}{2}\rfloor$ when $N$ is odd and $k_0=-\lfloor\frac{N}{2}\rfloor+1$ if $N$ is even and $\phi_k$ are eigenfunction of $-\Delta$. 
For convenience we assume from now on that $N$ is odd.
The following  lemma gives bounds on the projection error, see~\cite[Lemma 1.94]{lord_powell_shardlow_2014}.
\begin{lemma}\label{eq:proj}
    For any $N\in\N$, $\alpha,\beta\in\R$, with $\beta\geq\alpha$, and $u\in\dot{H}^{\beta}$, there is a constant  $C>0$, such that,
    $$\|(I-P_N)u\|_{\dot{H}^{\alpha}}\leq CN^{\alpha-\beta}\|u\|_{\dot{H}^{\beta}}.$$
\end{lemma}
The spectral Galerkin approximation for \eqref{eq:CGLE} is defined as
\begin{equation}\label{eq:spec_CGLE}
du_N = \left[A u_N+ Ru_N - (1+i\mu)P_N(u_N|u_N|^2)\right]dt+\sigma dW_N,  \quad 
u_N(0)=P_N u_0,
\end{equation}
where $W_N:=P_NW.$
Recall~\eqref{eq:A defn} and note that
$P_NAu=AP_Nu, \quad P_Ne^{tA}u=e^{tA}P_Nu.$
Then  the mild formulation for \eqref{eq:spec_CGLE} is given by
\begin{equation}\label{eq:spec_mild}
\begin{aligned}
u_N(t)&=e^{t A }P_Nu_0+\int_0^t e^{(t-s)A} \Big[Ru_N(s)+(1+i\mu)P_N(|u_N(s)|^2u_N(s))\Big]ds\\
& + \sigma \int_0^te^{(t-s)A} dW_N(s).
\end{aligned}
\end{equation}
We define the stochastic convolution, $\eta_N$, for $W_N$, by 
$$
\eta_N(t) := \sigma \int_0^te^{(t-s)A} dW_N(s).
$$
Thus, using~\eqref{def_stoch_conv}, we have 
\begin{equation*}\begin{split}\eta_N(t)&=\sigma P_N\sum_{k\in\Z}\sqrt{q_k}\left(\int_0^t e^{(t-s)A}d\beta_k^r(s)+i\int_0^t e^{(t-s)A}d\beta_k^i(s)\right)\phi_k=P_N\eta(t).\end{split}\end{equation*}

\begin{theorem}\label{Existence}
   For $r\geq0$ and $u_0\in \dot{H}^1$,  there exists a unique mild solution of \eqref{eq:RPDE}  satisfying the following estimate, for a constant $C_{T}>0$,
    $$
    \begin{aligned} 
    \sup\limits_{t\in [0,T]}\!\|v(t)\|_{\dot H^1}^2+ \|v\|^2_{L^2(0,T;\dot H^2)}\leq & C_T\Big[1+ \|u_0\|^2_{\dot{H}^1}+\sup_{t\in[0,T]}\!\|\eta(t)\|_{L^2}^{12}+ \|\eta\|^{12}_{L^2(0,T; \dot H^1)}\Big].
    \end{aligned} 
    $$
\end{theorem}
\begin{proof}
Consider a sequence of the spectral Galerkin approximations of the solution of~\eqref{eq:RPDE},  similar to~\eqref{eq:spec_CGLE}. The local Lipschitz continuity of $\Psi_0$ ensures the existence and uniqueness of solutions for the resulting system of random ODEs. The  estimates in Lemma~\ref{energy}, can be shown to hold for the spectral Galerkin approximations, which give uniform bounds for $v_N$  in $ L^{\infty}(0,T;\dot{H}^1)\cap L^{2}(0,T;\dot{H}^2)$ and for $\partial_t v_N$  in $L^2(0,T;L^2)$.
Hence, upto a subsequence not relabeled, $v_N \rightharpoonup v$   weakly in $L^{2}(0,T;\dot{H}^2)$  and $\partial_t v_{N} $ converges weakly in $L^2(0,T;L^2)$, see e.g.~\cite{Evans2010PartialDE}. 
The Aubin-Lions lemma, see e.g.~\cite{Aubin,Serrano2013}, implies strong convergence  $v_N \to v$ in $L^2(0,T;\dot{H}^{1})$.
To show that $v$ is a mild solution of~\eqref{eq:RPDE} we use that $v_N$ is a mild solution to the spectral Galerkin approximation  and  consider 
\begin{equation*}
\begin{aligned} 
&\lim_{N\to\infty}\Big\|v_{N} -e^{tA}u_0-\int_0^t e^{(t-s)A}\Psi_0(v(s)+\eta(s))ds\Big\|\\&=\lim_{N\to\infty}\Big\|e^{tA} P_{N} u_0 -e^{tA} u_0
+ \int_0^t \!\! e^{(t-s)A}\big(P_{N}\Psi_0(v_{N}(s)+\eta_{N}(s)) - \Psi_0(v(s)+\eta(s))\big)ds\Big\|\\
&\leq\lim_{N\to\infty}\big\| (I-P_{ N}) u_0\big\| + \lim_{N\to\infty}\int_0^t \big\|P_{N}\Psi_0(v_{N}(s)+\eta_{N}(s))-\Psi_0(v(s)+\eta(s))\big\|ds =: I + II.
    \end{aligned}
\end{equation*}
Using that $u_0\in \dot{H}^1$ and Lemma~\ref{eq:proj} yields
$$
0 \leq I \leq \lim_{N\to\infty}N^{-1}\| u_0\|_{\dot{H}^{1}} =0.
$$
The second term can be estimated as 
\begin{equation*}
    \begin{split}
  II 
&\leq\lim_{N\to\infty}\int_0^t \big\|P_{N}\Psi_0(v_{N}(s)+\eta_{N}(s))-P_{N}\Psi_0(v(s)+\eta(s))\big\| ds\\
&+\lim_{N\to\infty}\int_0^t \big \|(I-P_{N})\Psi_0(v(s)+\eta(s)) \big\|ds =: J + \tilde J.
    \end{split}
\end{equation*}
Since $v+\eta \in C([0,T];\dot{H}^1)$ we can apply Lemma~\ref{eq:proj}  to conclude  
$\tilde J = 0$.
Using inequality~\eqref{eq:loc_lip_L2} and  Lemma~\ref{stoch_app}   yields
\begin{equation*}
    \begin{split}
J & \leq\!\! \lim_{N\to\infty}\!\!\int_0^t\!\!\! \!\big[1+\|v_N(s)+\eta_{N}(s)\|^2_V+\|v(s)+\eta(s)\|^2_V\big]\|v_N(s)+\eta_{N}(s)-v(s)-\eta(s)\|  ds \\
%&\leq\lim_{N\to\infty}\sup_{t\in[0,T]}(1+\|v_N(t)+\eta_{N}\|^2_V+\|v(t)+\eta\|^2_V)\int_0^t \|v_N(t)+\eta_{N}-v(t)-\eta(t)\|ds\\
&\leq\lim_{N\to\infty}C_T\int_0^t \big(\|v_N(s)-v(s)\|+\|\eta_{N}(s)-\eta(s)\|\big)\, ds.
    \end{split}
\end{equation*}
The strong convergence in $L^{2}(0,T; L^2)$ implies  $\|v_N-v\|_{L^2((0,T)\times \T)} \to 0$ as $N\to \infty$, for $t\in [0,T]$. Using that $\eta\in L^{\infty}(0,T; \dot{H}^1)$ and applying  Lemma~\ref{eq:proj} yields  $\|\eta_{N}(t)-\eta(t)\| \to 0$  as $N \to \infty$.  Thus we have that $J = 0$ and hence $II=0$. Together with $I=0$ we obtain that $v$ is a mild solution of~\eqref{eq:RPDE}.

To show  uniqueness, we assume there exist two solutions, $v_1$ and $v_2$. Then testing the equation for the difference $v_1-v_2$
with $\overline v_1-\overline v_2$ and integrating over $\T$ yield
\begin{equation*}
    \begin{aligned}
    &  \frac 12 \frac{d}{dt}\|v_1(t)-v_2(t)\|^2+\|\nabla (v_1(t)-v_2(t))\|^2\\
      %=\big\langle \Psi_0(v_1(t)+\eta(t))-\Psi_0(v_2(t)+\eta(t)), v_1(t)- v_2(t)\big\rangle_{\R}\\
& \quad \leq \|\Psi_0(v_1(t)+\eta(t))-\Psi_0(v_2(t)+\eta(t))\|\|v_1(t)-v_2(t)\|\leq C\|v_1(t)-v_2(t)\|^2,
    \end{aligned}
\end{equation*}
where we used \eqref{eq:loc_lip_L2} and   $v_1, v_2, \eta \in C([0,T],V)$. Then Grönwall's inequality ensures
$$\|v_1(t)-v_2(t)\|^2\leq e^{Ct}\|v_1(0)-v_2(0)\|^2=0,$$
 which implies that $v_1 = v_2$ a.e. in $(0,T)\times \T$ and a.s. in $\Omega$.
\end{proof}
Now we prove higher regularity properties for $v$.
\begin{lemma}\label{mom_u}
For $t\in[0,T]$ and $u_0\in{\dot H}^\alpha$, with $\alpha\in[1,2)$, the mild solution of \eqref{eq:RPDE} satisfies 
    $$\|v(t)\|_{\dot{H}^{\alpha}}\leq C_T\big(1+ \|u_0\|_{{\dot H}^\alpha} +\|u_0\|^3_{{\dot H}^1} + \sup_{s\in[0,T]}\|\eta(s)\|^{10}_{\dot H^1}\big), $$
  where $C_T$ is a positive constant.
\end{lemma}
\begin{proof}
Using the formula for the mild solution yields  
\begin{equation*}
    \begin{split}
        \|v(t)\|_{\dot{H}^{\alpha}}
    &\leq \|u_0\|_{\dot{H}^{\alpha}}+\int_0^t   \|\tilde{\Delta}^{\frac{\alpha}{2}}e^{(t-s)A}\|_{\mathcal{L}(L^2)}\|\Psi_0(v(s)+\eta(s))\|ds\\
    &\leq  \|u_0\|_{\dot{H}^{\alpha}}+C_T\int_0^t (t-s)^{-\frac{\alpha}{2}}\big(1+\|v(s)\|_{L^6}^3+\|\eta(s)\|^3_{L^6}\big)ds,
    \end{split}
\end{equation*}
where we used the estimates for the semigroup in Lemma~\ref{lem:semigroup} and that  $|\Psi_0(w)|^2\leq C(|w|^2 + |w|^6)$. Considering $\|w\|_{L^6}^3 \leq \|w\|_{L^2}^2 \|w\|_{\dot H^1}$ and applying the estimates of Theorem~\ref{Existence} yield the result.
\end{proof}
Using the regularity estimates for $v$ and for the stochastic convolution, we obtain the existence and uniqueness results and the moment bounds for solutions of~\eqref{eq:CGLE}.
\begin{theorem}\label{mom_us}
For $W(t)$ which is $\dot{H}^r-$valued, with $r\geq0$, and $u_0\in\dot{H}^\alpha$, with  $\alpha\in[1,r+1]\cap[1,2)$,  there exists a unique mild solution $u$ of \eqref{eq:CGLE} and for all $p\in[2,\infty)$ we have for a $C_{T,p}>0$,
   $$\mathbb E\big[\sup_{t\in[0,T]}\|u(t)\|_{\dot{H}^{\alpha}}^p\big]\leq C_{T,p}.$$
\end{theorem}
\begin{proof}
Note that $u:=v+\eta$ is a solution of \eqref{eq:CGLE} and due to the uniqueness of $v$ and $\eta$, we have the uniqueness of $u$. The moment bound is obtained by using Lemma~\ref{mom_u} to estimate $u$ a.s. in terms of the stochastic convolution and then use the moment bound for $\eta$ proven in Lemma~\ref{stoch_app}.
\end{proof}
Following similar calculations, we obtain moment bounds for the spectral Galerkin approximation.
\begin{theorem}\label{mom_u_Ns}
For $W(t)$ which is $\dot{H}^r-$valued, with $r\geq 0$, and $u_0\in\dot{H}^\alpha$, with $\alpha\in[1,r+1]\cap[1,2)$,  there exists a unique mild solution $u_N$  for~\eqref{eq:spec_CGLE}, for all $N\in\N$. Further for all $N\in\N$ and  for all $p\in[2,\infty)$ there is a $C_{T,p}>0$ such that 
   $$\E \big[\sup_{t\in[0,T]}\|u_N(t)\|_{\dot{H}^{\alpha}}^p\big]\leq C_{T,p}.$$
\end{theorem}
\begin{proof}
The proof of this theorem follows  the same lines as the proof as Theorem~\ref{mom_us} and for $u_N$ we have the same estimates as in Lemmas~\ref{energy}, and~\ref{mom_u}.
\end{proof}

\section{Numerical discretization}\label{sec_Num}
\subsection{Numerical method}
The full-discretized numerical scheme, considered in this section, is a spectral Galerkin approximation in space %discretization 
and a splitting method for the time discretization. The spectral Galerkin approximation is defined in \eqref{eq:spec_CGLE}. For the time discretization we consider the first-order Lie-Trotter splitting, where we split the lower-order terms from the diffusion and stochastic term, as proposed by \cite{brehier_analysis_2019} for the Allen-Cahn equation. We take a uniform time discretization of the interval $[0,T]$ with step $\dt$ and let $t_m=m\dt$, for $m=1,\ldots,M$ where $M=T/\dt$.

In the formulation of the splitting method we use  the flow function, $\Phi_{\dt}$, of the complex ODE 
$$
z'=Rz-(1+i\mu)|z|^2z, 
$$ 
which is defined as
\begin{equation}\label{eq:Phi}
\begin{aligned}
   & \Phi_{\dt}(z):=\sqrt{\frac{R}{|z|^2-e^{-2R\dt}(|z|^2-R)}}\exp\left[-i\frac{\mu}{2}\ln\left[1+|z|^2\frac{e^{2R\dt}-1}{R}\right]\right]z, &&  \dt >0,\\
   & \Phi_0(z):=z, &&  \dt =0.
    \end{aligned} 
\end{equation}

The splitting method is determined by the following steps: first for each $x \in \mathbb T$ solve the nonlinear complex ODE
\begin{equation}\label{eq:phi_PDE}
    \begin{cases}
        d\tilde{u}_{\dt}(x,t)=R\tilde{u}_{\dt}(x,t)-(1+i\mu)|\tilde{u}_{\dt}(x,t)|^2\tilde{u}_{\dt}(x,t)&  \text{ in } \; \T\times(t_m,t_{m+1}],\\
        \tilde{u}_{\dt}(x,t_m)=U^m_{\dt}, 
    \end{cases}
\end{equation}
followed by the linear SPDE 
\begin{equation}\label{OU_eq}
    \begin{cases}
        d\hat{u}_{\dt}(x,t)=(1+i\nu)\Delta\hat{u}_{\dt}(x,t) dt+\sigma dW(t)& \; \; \text{ in } \; \T\times (t_m,t_{m+1}], \\
        \hat{u}_{\dt}(x,t_m)= \tilde{u}_{\dt}(x,t_{m+1}).
    \end{cases}
\end{equation}
The update for the next time-step is then determined from  
$$U^{m+1}_{\dt}=\hat{u}_{\dt}(t_{m+1}),$$
where $U^m_\dt(x)\approx u(x,t_m).$
For the full discretization $U^m_{N,\dt}$ we project onto the $N$ dimensional subspace.
Since \eqref{OU_eq} is a linear SPDE for which  the explicit solution is known, using the flow function $\Phi_{\dt}$ in~\eqref{eq:Phi}, a time step of the fully-discretized problem is given by
\begin{equation}\label{eq:Full}
U^{m+1}_{N,\dt}=e^{\dt A}P_N\Phi_{\dt}(U_{N,\dt}^m)+\sigma\int_{t_m}^{t_{m+1}}e^{(t_{m+1}-s)A}dW_N(s).
\end{equation}
The splitting method in \eqref{eq:Full}  can be written as
$$U^{m+1}_{N,\dt}=e^{\dt A}U_{N,\dt}^m+\dt  \, e^{\dt A}\frac{P_N\Phi_{\dt}(U_{N,\dt}^m)-U_{N,\dt}^m}{\dt}+\sigma\int_{t_m}^{t_{m+1}}e^{(t_{m+1}-s)A}dW_N(s), $$
which, similar to \cite{brehier_analysis_2019} for the Allen-Cahn equation,  can be interpreted as the accelerated Euler method of the auxiliary SPDE 
\begin{equation}\label{eq:auxeq}
du_{N,\dt}=\big((1+i\nu)\Delta u_{N,\dt}+P_N \Psi_{\dt}(u_{N,\dt})\big)dt+\sigma dW_N,    
\end{equation}
where,
\begin{equation}\label{eq:psi}
\begin{aligned}
  & \Psi_{\Delta t}(z):=\frac{\Phi_{\dt}(z)-z}{\dt}\\
  & =\frac{1}{\Delta t}\bigg[\frac{e^{R\dt}}{\sqrt{|z|^2\alpha(\dt)+1}}\exp\left({-i\frac{\mu}{2}\ln\Big[|z|^2\frac{e^{2R\Delta t}-1}{R}+1\Big]}\right)-1\bigg]z,  &&  \dt>0, \\
 &\Psi_{0}(z):=  Rz-(1+i\mu)|z|^2z, &&  \dt=0,
\end{aligned}
\end{equation}
where $\alpha(\dt):= (e^{2R\Delta t}-1)/R$.
To prove error estimates for the numerical scheme we need the well-posedness results and moment bounds for solutions of both the continuous and discrete problems. Theorems~\ref{mom_us} and~\ref{mom_u_Ns} and   Lemma~\ref{stoch_app} imply the well-posedness and moment bounds for $u$, $u_N$, and $\eta$ in the case of $\dot{H}^r-$valued noise with $r\geq 0$. 
We  need to prove moment bounds for $u_{N,\dt}$ and $U^{m}_{N,\dt}$, see Section~\ref{sec:momentbounds}.

First, however, we prove two preliminary lemmas that give estimates for $\Psi_{\dt}$ and $\Phi_{\dt}$, which  are required  for the proof of the  moment bounds for $u_{N,\dt}$ and $U^{m}_{N,\dt}$ and the derivation of the error estimates.
\begin{lemma}\label{cubic}
For all $\dt \in  [0,1)$, the following inequalities hold 
$$
\begin{aligned} 
(i) \quad  & |\Phi_{\dt}(z)|\leq  e^{R\dt}|z|, \qquad 
\qquad (iii)  \quad  \langle \Psi_{\dt}(z), z\rangle_\R  \leq C |z|^2, \\
(ii) \quad & |\Psi_{\dt}(z)|\leq C(1+|z|^3),  \qquad  (iv) \quad \langle \Psi_{\dt}(z_1+z_2), z_1\rangle_{\R}  \leq C(1+|z_1|^2+|z_2|^4), 
\end{aligned} 
$$  
for all $z, z_1,z_2\in\C$,    where the constant $C>0$ is independent of $\dt$.
\end{lemma}
\begin{proof}
For $(i)$, from the explicit formula for $\Phi_{\dt}$ in \eqref{eq:Phi} we have
\begin{equation}\label{eq:linear_phi}
    |\Phi_{\dt}(z)|\leq \bigg|\frac{e^{R\dt}}{\sqrt{\alpha(\dt) |z|^2 +1}} e^{ - i \frac \mu 2 \ln(1 + \alpha(\dt)|z|^2)}\bigg| |z | \leq  e^{R\dt}|z|, 
\end{equation}
where $\alpha(\dt) = (e^{2R\dt}-1)/R\geq 0$. The inequalities $(ii)$-$(iv)$  for $\dt=0$ can be easily checked and  we focus  on $\dt>0$. The function $\Psi_{\dt}(z)$ can be rewritten as
$$\Psi_{\dt}(z)=\frac{\int_0^{\dt}\frac{d}{ds}\Phi_{s}(z)ds}{\dt}=\frac{\int_0^{\dt}\Psi_{0}(\Phi_{s}(z))ds}{\dt}.$$
Hence, using~\eqref{eq:linear_phi} for $\Phi_{\dt}$  and cubic non-linearity of $\Psi_0$ we obtain $(ii)$, i.e.,
$$
|\Psi_{\dt}(z)|\leq \sup_{s\in[0,\dt]}|\Psi_{0}(\Phi_s(z))|\leq C  \Big(1+\sup_{s\in[0,\dt]}|\Phi_s(z)|^3\Big)\leq C \big(1+|z|^3\big).
$$
To show the third estimate $(iii)$ we consider 
 $$\langle  \Psi_{\dt}(z), z\rangle_\R %=\Big\langle \frac{\Phi_{\dt}(z)-z}{\dt},  z\Big\rangle_\R
 = \frac{1}{\dt}\big(\langle \Phi_{\dt}(z), z\rangle_\R-|z|^2\big)\leq \frac{1}{\dt}(e^{R\dt} -1)|z|^2 \leq C |z|^2,$$
 where we used that $0 \leq (e^{R\dt} -1)\leq C\dt$ for  $\dt \in [0,1)$. 

The proof of $(iv)$ is divided into three cases. 
For $|z_1|^2\leq 4R$ we have 
$$\langle \Psi_{\dt}(z_1+z_2), z_1\rangle_{\R}\leq |\Psi_{\dt}(z_1+z_2)||z_1|\leq C(1+|z_1|^3+|z_2|^3)\leq C(1+|z_2|^3).
$$
 For $|z_1|^2>4R$ and $|z_2|\geq\frac{1}{2}\min\{1,\frac{1}{\mu}\}|z_1|$ we have
$$\langle \Psi_{\dt}(z_1+z_2), z_1\rangle_{\R}\leq |\Psi_{\dt}(z_1+z_2)||z_1|\leq C(1+|z_1|^3+|z_2|^3)|z_2|\leq C(1+|z_2|^4).$$
The last case is $|z_1|^2>4 R$ and $|z_2|<\frac{1}{2}\min\{1,\frac{1}{\mu}\}|z_1|$. Note that $|z_1+z_2|\geq \sqrt{R}$.
For $z_1, z_2 \in \mathbb C$ we can write $z_1 = \tilde z_1 e^{i\omega}$ and $z_2 = \tilde z_2 e^{i\omega}$, with $\tilde z_1 >0$ and $\tilde z_2 \in \mathbb C$.  Then 
$$\langle \Psi_{\dt}(z_1+z_2), z_1\rangle_{\R} =\langle e^{i\omega}\Psi_{\dt}(\tilde z_1+\tilde z_2), e^{i\omega} \tilde z_1\rangle_{\R} = \langle \Psi_{\dt}(\tilde z_1+\tilde z_2),  \tilde z_1\rangle_{\R}
$$
and, without loss of generality, we can assume  that $z_1>0$. 
From~\eqref{eq:psi} we have
\begin{equation}\label{eq:psi_rep}
\Psi_{s}(z_1+z_2) 
=\frac{\int_0^{s} \frac{d}{d\tau}\Phi_{\tau}(z_1+z_2) d\tau}{s}=\frac{\int_0^{s} \Psi_0\left(\Phi_{\tau}(z_1+z_2)\right)d\tau}{s}
\end{equation}
and consider the two cases $\Re(\Psi_{s}(z_1+z_2))>0$ and   $\Re(\Psi_{s}(z_1+z_2)) <0$, with
$$\mathcal{C}:=\{s\in[0,\dt], z_1>2\sqrt{R}, |z_2|< \frac 1{2} \min\{1, \frac 1\mu\} z_1 :\Re\left(\Psi_{s}(z_1+z_2)\right)\geq 0\} $$ 
 and 
$$\int_{[0,\dt]\cap \mathcal{C}} \Re\left(\Psi_0\left(\Phi_{s}(z_1+z_2)\right)\right)ds\geq 0,
\quad \quad \int_{[0,\dt]\cap \mathcal{C}^c} \Re\left(\Psi_0\left(\Phi_{s}(z_1+z_2)\right)\right)ds < 0.$$
 Thus, it is sufficient to find an upper bound for $\langle\Psi_0(\Phi_s(z_1+z_2)),z_1\rangle_{\R}$ for $(s,z_1, z_2)\in\mathcal{C}$.
 The first equality in~\eqref{eq:psi_rep} implies that the real part of $\Psi_{s}(z_1+z_2)$ is positive only when $\Re(\Phi_{s}(z_1+z_2)) \geq \Re(z_1 + z_2)$. This, together with  $|\Phi_s(z_1,z_2)|\leq |z_1+z_2|$ for $|z_1+z_2|\geq \sqrt{R}$, implies 
 \begin{align*}
   |\Phi_s(z_1,z_2)-z_1|^2 %&= \big(\Re(\Phi_s(z_1,z_2))-z_1\big)^2+\Im(\Phi_s(z_1,z_2))^2  \\
   &= z_1^2-2z_1\Re(\Phi_s(z_1,z_2))+\Re(\Phi_s(z_1,z_2))^2+\Im(\Phi_s(z_1,z_2))^2 \\
   &\leq z_1^2-2z_1(z_1+\Re(z_2))+|z_1+z_2|^2\\
   &= -z_1^2-2z_1\Re(z_2)+z_1^2+2z_1\Re(z_2)+\Re(z_2)^2+\Im(z_2)^2=|z_2|^2.
 \end{align*}
Hence, $\Phi_s(z_1,z_2)\in B_{z_1}(|z_2|)$  and  $|\Im(\Phi_s(z_1+z_2)) |\leq |z_2|$ for  $(s, z_1, z_2 )\in\mathcal{C}$.
 Thus
$$
\begin{aligned}
\langle\Psi_0(\Phi_s(z_1+z_2)),z_1\rangle_{\R}
&\leq R(z_1 + |z_2|) z_1 - 
z_1^2|\Phi_s(z_1+z_2)|^2  + |z_2| z_1 |\Phi_s(z_1+z_2)|^2 \\
&+ \mu  \Im (\Phi_s(z_1+z_2)) z_1|\Phi_s(z_1+z_2)|^2 
\leq R(z_1 + |z_2|) z_1,
\end{aligned}
$$
where we used that $|z_2|+\mu|z_2|\leq  z_1$. 
\end{proof}

\begin{lemma}\label{Psi_loc_lip}
 For all $\dt\in [0,1)$
 we have the following estimates
    \begin{equation} \label{Psi_propert}
    \begin{aligned}
\text{ (i) } &  |\Psi_{\dt}(z_1)-\Psi_{\dt}(z_2)|\leq  C\big(1+|z_1|^2 + |z_2|^2 \big)|z_1-z_2| \quad \text{ for } z_1,z_2\in\C, \\
\text{ (ii) } & |\Psi_{\dt}(z)-\Psi_0(z)|\leq  C \dt  e^{R}\big[|z|+  e^{4R} |z|^5\big]
    \quad \text{ for } z \in \mathbb C,
    \end{aligned}
    \end{equation}
    where the constant $C>0$ is independent of $\dt$.
\end{lemma}
\begin{proof}
$(i)$ In the case of $\dt=0$, one can use \eqref{eq:loc_lip_L2}. Using the formula \eqref{eq:psi} for $\dt>0$ yields 
$$
\begin{aligned}
 \Psi_{\Delta t}(z_1)-\Psi_{\Delta t}(z_2)
&= \frac{z_1-z_2}{\Delta t}\Big[\sqrt{\frac{R\alpha(\dt)+1}{\alpha(\dt) |z_1|^2+1}}\exp\Big({-i\frac{\mu}{2}\ln\left[\alpha(\dt)|z_1|^2+1\right]}\Big) - 1\Big] \\
& \qquad + \frac{z_2}{\Delta t}\left[
\sqrt{\frac{R\alpha(\dt)+1}{\alpha(\dt) |z_1|^2+1}}\exp\left({-i\frac{\mu}{2}\ln\big[\alpha(\dt)|z_1|^2+1\big]}\right) \right.\\
& \qquad - \left. \sqrt{\frac{R\alpha(\dt)+1}{\alpha(\dt) |z_2|^2+1}}\exp\left({-i\frac{\mu}{2}\ln\big[\alpha(\dt)|z_2|^2+1\big]}\right) \right]
=: I + II,
\end{aligned}
$$
For $I$, using the intermediate value theorem and  $0<\alpha(\dt) \leq C \dt$, yields   
$$\left|\sqrt{\frac{R\alpha(\dt)+1}{\alpha(\dt) |z_1|^2+1}}\exp\left({-i\frac{\mu}{2}\ln\big[\alpha(\dt) |z_1|^2+1\big]}\right) - 1\right| \leq C \dt(1+ |z_1|^2).$$
For the second term $II$, using the Lipschitz continuity of $x^{-1/2}$ for $x\geq 1$ we obtain 
$$
\begin{aligned} 
\Big|\exp\left({-i\frac{\mu}{2}\ln\big[\alpha(\dt)|z_1|^2+1\big]}\right)\Big|\left|\sqrt{\frac{R\alpha(\dt)+1}{\alpha(\dt) |z_2|^2+1}}- 
\sqrt{\frac{R\alpha(\dt)+1}{\alpha(\dt) |z_1|^2+1}}\right|
\\
\leq \frac 12 e^{R\dt} \alpha(\dt) \big||z_2|^2 -  |z_1|^2\big|
\leq C e^{R\dt} \dt \big(|z_1|+ |z_2|\big)|z_1 - z_2|
\end{aligned} 
$$
and 
$$
\begin{aligned} 
\sqrt{\frac{R\alpha(\dt)+1}{\alpha(\dt) |z_2|^2+1}}
\Bigg|\exp\left({-i\frac{\mu}{2}\ln\big[\alpha(\dt)|z_2|^2+1\big]}\right) - \exp\left({-i\frac{\mu}{2}\ln\big[\alpha(\dt)|z_1|^2+1\big]}\right)\Bigg|
\\
\leq \frac \mu 2 e^{R\dt} \alpha(\dt) \big||z_2|^2 -  |z_1|^2\big|
\leq C e^{R\dt} \dt \big(|z_1|+ |z_2|\big)|z_1 - z_2|. 
\end{aligned}
$$
Combining the estimates from above implies the result stated in the lemma. 

$(ii)$ The proof  of the second estimate  is based on a proof for a similar results for the Allen-Cahn equation in~\cite{brehier_analysis_2019}.
 Using the formulas for   $\Psi_{\dt}$  and $\Psi_0$ we can write 
    \begin{equation*}
    \begin{split}
      \Psi_{\dt}(z)-\Psi_0(z)&=\frac{z}{\Delta t}\left[\frac{e^{R\Delta t}}{\sqrt{\alpha(\dt)|z|^2+1}}\exp\left({-i\frac{\mu}{2}\ln\left[\alpha(\dt)|z|^2+1\right]}\right)-1\right] \\
      & \quad-Rz+(1+i\mu)|z|^2z =\frac{z}{\dt}\big(g_z(\dt)-g_z(0)-\dt g_z'(0)\big),
    \end{split}
    \end{equation*}
    with 
    $$g_z(t)=\frac{e^{R t}}{\sqrt{\alpha(t)|z|^2+1}}\exp\left({-i\frac{\mu}{2}\ln\left[\alpha(t)|z|^2+1\right]}\right) \;  \text{ and }  \; g'_z(0)=R-(1+i\mu)|z|^2, 
    $$
    where $\alpha(t) = (e^{2Rt}-1)/R$. 
    The first and second derivatives of $g_z(t)$, for all $t\geq 0$, are given by 
    \begin{equation*}
    \begin{split}
      g'_z(t)
        %&=g_z(t)\left(\frac{R(Re^{-2Rt}-|z|^2e^{-2Rt})}{|z|^2(1-e^{-2Rt})+Re^{-2Rt}}-i\mu\frac{|z|^2e^{2R\dt}}{|z|^2\frac{e^{2R\Delta t}-1}{R}+1}\right)\\
         &=g_z(t)\left(R-\frac{|z|^2 e^{2Rt}}{\alpha(t)|z|^2+1}-i\mu\frac{|z|^2e^{2R t}}{\alpha(t)|z|^2 +1}\right) =g_z(t)\left(R-\frac{(1+ i \mu) |z|^2 e^{2Rt}}{\alpha(t)|z|^2+1}\right) , \\
            g_z''(t)&=g_z(t)\Big[\Big[R-(1+ i \mu)\frac{ |z|^2 e^{2Rt}}{\alpha(t)|z|^2+1}\Big]^2 \\
           & \qquad + (1+ i \mu)\frac{2Re^{2Rt}|z|^2\big[|z|^2(1-(1-1/R)e^{2Rt}) -1\big] }{(\alpha(t) |z|^2+ 1)^2}\Big].
        \end{split}
    \end{equation*}
  We have that  for all $\dt\in[0,1)$,
    $$
    \begin{aligned} 
    |g_z(\dt)|&\leq  e^{R}, \qquad 
    |g_z'(\dt)|\leq  e^{R}\big[R + (1+ |\mu|)e^{2R} |z|^2\big], \\
    |g_z''(\dt)| &\leq 2e^{R}\Big[R^2 + \Big(\big[(1+|\mu|)^2 + 1+|\mu|\big] e^{4R} +  R(1+|\mu|)e^{2R}\Big) |z|^4\Big] .
    \end{aligned} 
    $$
    Applying Taylor's expansion to $g_z(t)$ around zero  yields the estimates.
\end{proof}

\subsection{Moment bounds} \label{sec:momentbounds}
Next we prove the moment bounds for $u_{N,\dt}$ and $U^{m}_{N,\dt}$. 
\begin{lemma}\label{mom_aux}
Let $W(t)$ be $\dot{H}^r-$valued, with $r\geq 0$,  and $u_0\in\dot{H}^\alpha$, where $\alpha\in[1,r+1]\cap[1,2)$.
Then there exists a unique mild solution $u_{N,\dt}$ of the auxiliary SPDE~\eqref{eq:auxeq}  such that 
   $$\sup_{t\in[0,T]}\|u_{N,\dt}(t)\|_{L^p(\Omega,\dot{H}^{\alpha}(\T))}\leq C_{T,p},$$
  for all  $p\in[2,\infty)$ and  the constant $C_{T,p}>0$.
\end{lemma}
\begin{proof} The proof follows the same step as the proof of the corresponding results for the SCGLE~\eqref{eq:CGLE} in   Lemmas~\ref{Existence}, and~\ref{energy}. 
Notice that to show the well-posedness results and derive the a priori estimates we used the local  Lipschitz continuity and cubic nonlinearity of  $\Psi_0$. 
The same properties are shown for~$\Psi_{\dt}$, see Lemmas~\ref{cubic} and~\ref{Psi_loc_lip}. 
\end{proof}
\begin{lemma}\label{full_disc_moment_b}
 Let $W(t)$ be $\dot{H}^r-$valued with $r\geq 0$ and $u_0\in\dot{H}^\alpha$, with $\alpha\in[1,r+1]\cap[1,2)$. Assume $|\nu| \leq \sqrt{3}$ and $N^{-1}\leq C \dt^{1/2+\varsigma}$. Then for some $\varsigma >0$, there is a constant $C_T>0$ such that  
\begin{equation}
\sup_{0\leq m\leq M}\E\big[ \|U_{N,\dt}^m\|^2\big]+ \E\big[\sup_{0\leq m\leq M} \|U_{N,\dt}^m\|^4_{L^4}\big] \leq C_T.
\end{equation} 
\end{lemma}
\begin{proof}
Considering the $L^2$-norm of $U^{m+1}_{N,\dt}$ and using equation \eqref{eq:Full} yields 
\begin{equation}\label{eq_121}
\begin{aligned}
 \|U^{m+1}_{N,\dt}\|^2
  &\leq \big\|\Phi_{\dt}(U^{m}_{N,\dt})\big\|^2 +\Big\langle e^{\dt A} P_N\Phi_{\dt}(U^{m}_{N,\dt} ), \sigma \int_{t_m}^{t_{m+1}}\!\!\! \!\!\! e^{ (t_{m+1}-s)A}  dW_N(s)\Big \rangle
 \\
&\quad   + \Big\langle e^{\dt A} P_N\Phi_{\dt}(U^m_{N, \dt} ), \sigma \int_{t_m}^{t_{m+1}} \!\!\!\!\!\! e^{ (t_{m+1}-s)A} dW_N(s) \Big\rangle  + \sigma^2\Big\|  \int_{t_m}^{t_{m+1}}\!\!\!\!\!\! e^{(t_{m+1}-s)A}  dW_N(s) \Big\|^2 ,
 \end{aligned}
\end{equation}
where we used  the contraction property of the semigroup $e^{tA}$ generated by $A =(1+i\nu)\Delta $.
 Applying \eqref{eq:linear_phi} for $\Phi_{\dt}$, taking the expectation of \eqref{eq_121}, and using the independence of $\Phi_{\dt}(U^{m}_{N,\dt} )$ and the stochastic integral,  implies
 \begin{equation}\label{eq_122}
\begin{aligned}
 \E \|U^{m+1}_{N,\dt}\|^2 \leq  e^{2 \dt R}\E \big\|U^{m}_{N,\dt}\big\|^2 + \sigma^2 \mathbb E\Big\|  \int_{t_m}^{t_{m+1}}\!\!\!\! e^{ (t_{m+1}-s)A}  dW_N(s) \Big\|^2.
 \end{aligned}
\end{equation}
The Burkholder–Davis–Gundy inequality and the properties of the semigroup ensure 
\begin{equation} \label{estim_conv_23}
\begin{aligned}
&\mathbb E \Big[\Big\| \int_{t_m}^{t_{m+1}}\!\!\!\! e^{(t_{m+1}-\tau)A}  dW(\tau)\Big\|^p_{L^p} \Big]
\quad \leq
C_p \Big(\int_{t_m}^{t_{m+1}}\!\!\!\! \big\| e^{(t_{m+1}-\tau)A}  B \tilde \Delta^{-\frac{r}{2}-\frac{1}{4}-\frac{\varepsilon}{4}}\big\|^2_{\mathcal L(L^2)} d\tau  \Big)^{\frac p  2}
\leq C_p \dt^{\frac p 2},
\end{aligned}
\end{equation} 
for $p \geq 2$. 
 Using \eqref{estim_conv_23} with $p=2$ in \eqref{eq_122} yields
 \begin{equation} \label{estim_12}
 \begin{aligned}
  \mathbb E\big[\|U^{m+1}_{N,\dt}\|^2\big] \leq e^{2R\dt}\mathbb E(\|U^{m}_{N,\dt}\|^2 ) + C \dt.
  \end{aligned}
 \end{equation}
 Iterating \eqref{estim_12} over $m$ implies the estimate for the second moment   $\mathbb E\big[\|U^{m}_{N,\dt}\|^2\big]$ for all $m$ from $1$ to $M$.
 Considering the $L^4$-norm  of \eqref{eq:Full} yields
\begin{equation*} \label{estim_212}
\begin{aligned}
  \big\|U^{m+1}_{N,\dt}\big\|^4_{L^4}
 &= \big\| e^{\dt A}  P_N \Phi_{\dt}(U^{m}_{N,\dt})\big\|^4_{L^4}  +   2 \sigma^2\Big\langle \big|e^{\dt A} P_N \Phi_{\dt}(U^{m}_{N,\dt})\big|^2,
 \big| \mathcal J_{W_N}^m\big|^2 \Big\rangle \\
 & \quad +  2 \sigma\Big\langle \big|e^{\dt A} P_N \Phi_{\dt}(U^{m}_{N,\dt})\big|^2 e^{\dt A} P_N \Phi_{\dt}(U^m_{N, \dt}) ,
  \mathcal J_{W_N}^m \Big\rangle
 \\
 & \quad +  \sigma^3 \Big \langle e^{\dt A} P_N \Phi_{\dt}(U^{m}_{N,\dt}),
 \big |\mathcal J_{W_N}^m\big|^2   \mathcal J_{W_N}^m \Big\rangle
 +  \sigma^3 \Big \langle e^{\dt A} P_N \Phi_{\dt}(U^{m}_{N,\dt}),
 \big |\mathcal J_{W_N}^m \big|^2  \mathcal J_{W_N}^m \Big\rangle
\\
& \quad  +
 \sigma^4 \big\|  \mathcal J_{W_N}^m\big\|^4_{L^4} = I_1 + I_2 +I_3 +I_4 + I_5 + I_6 + I_7,
\end{aligned}
\end{equation*}
where 
$$
\mathcal J_{W_N}^m = \int_{t_m}^{t_{m+1}} \!\!\!\! e^{(t_{m+1}-s) A}  dW_N(s).
$$
Notice that for $|\nu| \leq 2\sqrt{p-1}/(p-2)$ and $p>2$, the following estimate holds
$$
\| e^{t A} w\|_{L^p} \leq \|w\|_{L^p} \qquad \text{ for }  w \in L^p(\T). 
$$
This follows from the fact that for such $p$  and $\nu$ we have
$$
|{\rm Im} \langle \Delta u, |u|^{p-2} u \rangle | \leq \frac{ p-2}{ 2\sqrt{ p-1}} {\rm Re} \langle \Delta u, |u|^{p-2} u \rangle,
$$
and hence 
$
{\rm Re} \langle (1+ i \nu) \Delta u, |u|^{p-2} u \rangle \leq 0, 
$
see e.g.~\cite{Okazawa}. Thus for $p=4$ and $|\nu|\leq \sqrt{3}$ and using estimate for $\Phi_{\dt}$, we obtain  
$$
\begin{aligned}
I_1 &\leq \| e^{\dt A}  \Phi_{\dt}(U^{m}_{N,\dt})\|^4_{L^4} + \|e^{\dt A}(P_N -I)\Phi_{\dt}(U^{m}_{N,\dt}) \|^4_{L^4} 
\\
&\leq  e^{R \dt } \| U^{m}_{N,\dt}\|^4_{L^4} + \|e^{\dt A}(P_N -I)\Phi_{\dt}(U^{m}_{N,\dt}) \|^4_{L^4}. 
\end{aligned}
$$
The properties of the semigroup 
$$
\|e^{tA} w\|_{L^r} \leq C_{\beta,T} t^{-\frac d2(\frac 12 - \frac 1r) + \frac \beta 2}\|w\|_{H^\beta}
$$
and of  projection $P_N$, see Lemma~\ref{eq:proj}, imply the following estimate  
$$
\begin{aligned} 
 \|e^{\dt A}(P_N -I)\Phi_{\dt}(U^{m}_{N,\dt}) \|^4_{L^4} 
&\leq C_{\beta, T} \dt^{-1/2+2\beta}\|(P_N -I)\Phi_{\dt}(U^{m}_{N,\dt})\|^4_{H^\beta} \\
&\leq C_{\beta,T} \dt^{-1/2+2\beta} N^{4\beta} \|U^{m}_{N,\dt}\|^4_{L^2} \leq C_{\beta, T}\dt \|U^{m}_{N,\dt}\|^4_{L^4},
\end{aligned}
$$
for $N^{-1}\leq C \dt^{\frac 1 2 +\varsigma}$ and  $\beta<0$ such that $\varsigma |\beta| \geq 3/8$. 
To estimate the next five terms in the expression for $\|U^{m+1}_{N,\dt}\|^4_{L^4}$, we  apply  the H\"older inequality and the estimates for the semigroup and  $\Phi_{\dt}$ and obtain 
$$
\begin{aligned}
|I_2+I_3+I_4 + I_5+I_6| &\leq   C\Big[ \dt \big\| e^{\dt A} P_N \Phi_{\dt}(U^{m}_{N,\dt})\big\|^4_{L^4} +
\sigma^4  \frac {1} { \dt} 
 \big \|\mathcal J_{W_N}^m\big\|_{L^4}^4 \Big]\\
 & \leq C \Big[\dt \|U^{m}_{N,\dt}\|_{L^4}^4 + \sigma^4  \frac {1} { \dt} 
 \big \|\mathcal J_{W_N}^m\big\|_{L^4}^4\Big].
 \end{aligned}
$$
Combining the estimates from above yields 
$$
\begin{aligned} 
\big\|U^{m+1}_{N,\dt}\big\|^4_{L^4} \leq e^{R \dt } \| U^{m}_{N,\dt}\|^4_{L^4} + C \dt  \| U^{m}_{N,\dt}\|^4_{L^4}  +  C\sigma^4(1+ 1/\dt)  \big \|\mathcal J_{W_N}^m\big\|_{L^4}^4. 
\end{aligned} 
$$
Summing over $m=0, \ldots, \tilde m-1$, using the telescoping series property and $e^{R\dt} \leq 1 + C\dt$ for $\dt \leq 1$,  and applying the discrete Gr\"onwall inequality yields
$$
\begin{aligned} 
\big\|U^{\tilde m}_{N,\dt}\big\|^4_{L^4} \leq C \|U^{0}_{N,\dt}\|_{L^4}^4+C \sigma^4 \sum_{m=0}^{\tilde m-1} (1+ 1/\dt)  \big \|\mathcal J_{W_N}^m\big\|_{L^4}^4, 
\end{aligned} 
$$
 for all $\tilde m=1, \ldots M$. Taking the expectation and using \eqref{estim_conv_23} for $p=4$ we obtain
\begin{equation} \label{estim_222}
\begin{aligned}
 \mathbb E \big[\sup_{1\leq m\leq M} \|U^{m}_{N,\dt}\|^4_{L^4}\big]
 &\leq  C \|U^{0}_{N,\dt}\|^4_{L^4}+ C \sigma^4 \Big[1+\frac 1\dt\Big]   \sum_{i=0}^{M-1}  \mathbb E \big[\big \|\mathcal J_{W_N}^i\big\|_{L^4}^4\big] \\
& \leq C \|P_Nu_0\|^4_{L^4}+C\sigma^4 \Big[1+\frac 1\dt\Big] \sum_{i=0}^{M-1}
\Big(\int_{t_i}^{t_{i+1}}\!\!\! \big\| e^{ A(t_{i+1}-s)}  B \tilde \Delta^{-\frac{r}{2}-\frac{1}{4}-\frac{\varepsilon}{4}}\big\|^2_{\mathcal L(L^2)} ds  \Big)^{2}\\
 & \leq   C (1+ \dt),
\end{aligned}
\end{equation}
which, together with the estimate for the second moment, yields the result stated in the lemma. 
\end{proof}
\subsection{Convergence results}
With all the moment bounds in hand we can prove the following convergence results for the fully discretized scheme~\eqref{eq:Full}.  

\begin{theorem}\label{conv}
Let $W(t)$ be $\dot{H}^r-$valued with $r\geq 0$, $|\nu|\leq \sqrt{3}$ and $u_0\in\dot{H}^\alpha$, with $\alpha\in[1,r+1]\cap[1,2)$. Let $\tilde{\alpha}\in (0,\alpha)$, $\beta\in (0,\min(\frac{1}{2},\frac{\alpha}{2}))$. Then for all $\dt\in(0,1)$ and $N\in\N$ with  $N^{-1}\leq C \dt^{\frac 1 2 +\varsigma}$ for some $\varsigma>0$, the following estimate holds
    $$
    \E\Big[\sup_{0\leq m \leq M}\1_{\Omega_{L,\alpha}}\|u(t_m)-U_{N,\dt}^{m}\|^2\Big]\leq C_{T,L,\alpha,\beta}\big(\dt^{2\beta}+N^{-2\tilde{\alpha}}\big),$$
    where $C_{T,L,\alpha,\beta}>0$ and 
    $$\Omega_{L,\alpha}:=\{\omega\in\Omega: \; \mathrm{H}_{\alpha}\leq L \; \text{ for all }  N\in\N \},$$
    with
    $$\mathrm{H}_{\alpha}:=\sup_{t\in[0,T]}\left(\|u(t)\|_{\dot{H}^{\alpha}}^2+\|u_N(t)\|_{\dot{H}^{\alpha}}^2+\|u_{N,\dt}\|_{\dot{H}^{\alpha}}^2+\|\eta(t)\|_{\dot{H}^{\alpha}}^2\right)+\sup_{0\leq m \leq M} \|U_{N,\dt}^m\|_{L^4}^2.$$
\end{theorem}
The error estimate in Theorem~\ref{conv} also implies the lower bound for the convergence rate of the numerical approximation in the probability sense.

\begin{corollary}\label{end_result}
Under the assumptions in Theorem~\ref{conv} we have 
$$\lim_{\dt\rightarrow 0,N\rightarrow \infty}\Prob\Big(\sup_{0\leq m\leq M}\|U^m_{N,\dt}-u(t_m)\|\geq K [N^{-\tilde{\alpha}}+\dt^{\beta}]\Big)=0.$$   
\end{corollary}

\begin{proof}
Considering the set $\Omega_{L,\alpha}$, defined in Theorem~\ref{conv},  we can write 
\begin{equation*}
\begin{aligned} 
&  \Prob\Big(\sup_{0\leq m\leq M}\|U^m_{N,\dt}-u(t_m)\|\geq K (N^{-\tilde{\alpha}}+\dt^{\beta})\Big) \\
&  \leq \Prob\big(\Omega^c_{L,\alpha}\big)+\Prob\Big(\Omega_{L,\alpha}\cap\sup_{0\leq m\leq M}\|U^m_{N,\dt}-u(t_m)\|\geq K (N^{-\tilde{\alpha}}+\dt^{\beta})\Big) = I + II.
\end{aligned} 
\end{equation*}
Using  Markov's inequality and the moment bounds yield
\begin{equation*}
\begin{aligned} 
I= 
1-\Prob(\Omega_{L,\alpha})\leq \frac{\E(\mathrm{H}_{\alpha})}{L}\leq \frac{C_{T} }{L}.
      \end{aligned} 
\end{equation*}
For the second term, applying Markov's inequality and Theorem~\ref{conv} with $0<\tilde{\alpha}<\hat{\alpha}$ and $\hat{\beta}>\beta$, we obtain 
\begin{equation*}
\begin{aligned} 
  II    & \leq \frac{1}{K^2(N^{-2\tilde{\alpha}}+\dt^{2\beta})}\E\Big[\1_{\Omega_{L,\alpha}}\sup_{0\leq m\leq M}\|U^m_{N,\dt}-u(t_m)\|^2\Big] \\
  & \leq 
  \frac{C_{T, L, \hat \alpha, \hat \beta} (N^{-2\hat{\alpha}}+\dt^{2\hat{\beta}})}{K^2(N^{-2\tilde{\alpha}}+\dt^{2\beta})}
\leq 2\max\{\dt^{2(\hat{\beta}-\beta)},N^{2(\tilde{\alpha}-\hat{\alpha})}\}\frac{C_{T, L, \hat \alpha, \hat \beta}}{K^2}.
\end{aligned} 
\end{equation*}
Taking the limit $\dt\rightarrow 0$ and $N\rightarrow\infty$ yields 
$$
\limsup_{\dt\rightarrow 0,N\rightarrow \infty}\Prob\Big(\sup_{0\leq m\leq M}\|U^m_{N,\dt}-u(t_m)\|\geq K (N^{-\tilde{\alpha}}+\dt^{\beta})\Big)\leq \frac{C_{T}}{L},
$$
and $L$ can be chosen arbitrarily large.
\end{proof}

We divide the proof of Theorem~\ref{conv}, into three steps, corresponding to the three components of the error 
\begin{equation}\label{eq:split_error}
    \begin{split}
      \E\big[\1_{\Omega_{L,\alpha}}\|u(t_m)-U_{N,\dt}^{m}\|^2\big]&\leq \underbrace{\E\big[\1_{\Omega_{L,\alpha}}\|u(t_m)-u_N(t_m)\|^2\big]}_{\text{Spatial discretization error}}+\underbrace{\E\big[\1_{\Omega_{L,\alpha}}\|u_N(t_m)-u_{N,\dt}(t_m)\|^2\big]}_{\text{Auxiliary error}}\\
      & \quad  +\underbrace{\E\big[\1_{\Omega_{L,\alpha}}\|u_{N,\dt}(t_m)-U_{N,\dt}^m\|^2\big]}_{\text{Time discretization error}}  .
    \end{split}
\end{equation}
First, we estimate  the mean square error of the spatial discretization using the spectral Galerkin method. 
\begin{lemma}\label{Space_convergence}
    Under the assumptions in Theorem~\ref{conv}, there exists a positive constant $C_T>0$ such that for all $\tilde{\alpha}\in (0,\alpha]$,
    $$\E\Big[\sup_{t\in[0,T]}\|u(t)-u_{N}(t)\|^p\Big]\leq C_{T,\tilde\alpha}N^{-p\tilde{\alpha}}  \qquad \text{ for } \; p \in [2, \infty).$$
\end{lemma}
\begin{proof}
    Using Lemma~\ref{eq:proj}   and moment bounds for $u$ in $H^{\tilde{\alpha}}$ we obtain
     \begin{equation*}
        \begin{split}
    \E\Big[\sup_{t\in[0,T]}\|u(t)-u_{N}(t)\|^p\Big]
    &\leq\E\Big[\sup_{t\in[0,T]}\big\|(I-P_N)u(t)\big\|^p\Big]
    +\E\Big[\sup_{t\in[0,T]}\|P_Nu(t)-u_{N}(t)\|^p\Big]\\
    &\leq N^{-\tilde{\alpha} p}\E\Big[\sup_{t\in[0,T]}\|u(t)\|_{\dot{H}^{\tilde{\alpha}}}^p\Big]
     +\E\Big[\sup_{t\in[0,T]}\|P_Nu(t)-u_{N}(t)\|^p \Big] \\
     &\leq C_{T,\tilde\alpha}N^{-\tilde{\alpha} p}+\E \Big[\sup_{t\in[0,T]}\|P_Nu(t)-u_{N}(t)\|^p\Big]  .
    \end{split}
    \end{equation*} 
    Defining $e^N(t):=P_Nu(t)-u_{N}(t)$ and using  $e^N(0)=P_Nu(0)-u_N(0)=0$ and
     the formulas for  $P_N u(t)$ and $u_N(t)$ as mild solutions of the corresponding SPDEs, yield 
    $$e^N(t)=\int_0^te^{(t-s)A}\left[P_N \Psi_0(u(s))-P_N\Psi_0(u_N(s))\right]ds,
    $$
    which is a mild solution to
    \begin{equation}\label{eq_error_space}
    \frac{d}{dt}e^N=Ae^N +P_N\Psi_{0}(u)-P_N\Psi_{0}(u_N). 
    \end{equation}
    Using
    \begin{equation*}\frac{d}{dt}\|u(t)\|^{p}
    =p\|u(t)\|^{p-2}\Big\langle \frac{du}{dt}(t), u(t) \Big\rangle_\R,
    \end{equation*}
 together with $\langle e^N(t),-A e^N(t) \rangle_{\R} \leq 0$,  and testing \eqref{eq_error_space} by $p\|e^N(t)\|^{p-2}  e^N(t)$ imply
    \begin{equation}\label{eq:start_mu}
    \begin{split}
           \frac{d}{dt}\|e^N(t)\|^p = & p\|e^N(t)\|^{p-2}\big\langle e^N(t),-Ae^N(t)+P_N\Psi_{0}(u(t))-P_N\Psi_{0}(u_N(t))\big\rangle_\R\\
    \leq & p\|e^N(t)\|^{p-2}\langle e^N(t),P_N\Psi_{0}(P_Nu(t))-P_N\Psi_{0}(u_N(t))\rangle_\R\\
     &+p\|e^N(t)\|^{p-2}\langle e^N(t),P_N\Psi_{0}(u(t))-P_N\Psi_{0}(P_Nu(t))\rangle_\R.
    \end{split}
    \end{equation}
    Using properties of $\Psi_0$, see~\eqref{eq:loc_lip_L2} and~\eqref{eq:F_bound}, together with the Cauchy-Schwarz and Young inequalities, yield
    \begin{equation*}
    \begin{split}
    \frac{d}{dt}\|e^N(t)\|^p&\leq  p\|e^N(t)\|^{p-1}\big(\|\Psi_{0}(P_Nu(t))-\Psi_{0}(u_N(t))\| +\|P_N\Psi_{0}(u(t))-P_N\Psi_{0}(P_Nu(t))\| \big)\\
   & \leq  C\big(1+\|P_Nu(t)\|_V^{2p}+\|u_N(t)\|_V^{2p}\big)\|e^N(t)\|^{p}+ (p-1)\|e^N(t)\|^{p}\\
   &\quad+C\big(1+\|P_Nu(t)\|_V^{2p}+\|u(t)\|_V^{2p}\big)\|(I-P_N)u(t)\|^p
    \end{split}
    \end{equation*}
   Integrating the last inequality over time and using   Lemma~\ref{eq:proj} implies
    \begin{equation*}
    \begin{split}
           \|e^N(t)\|^p
           &\leq C\big(1+\sup_{t\in[0,T]}\|P_Nu(t)\|_V^{2p}+\sup_{t\in[0,T]}\|u_N(t)\|^{2p}_V\big)\int_0^t\|e^N(t)\|^{p}ds\\
           &\quad+C_T\big(1+\sup_{t\in[0,T]}\|P_Nu(t)\|_V^{2p}+\sup_{t\in[0,T]}\|u(t)\|_V^{2p}\big)\sup_{t\in[0,T]}\|u(t)\|_{\dot H^{\tilde{\alpha}}}^pN^{-p\tilde\alpha}.
    \end{split}
    \end{equation*}
    Applying Grönwall's inequality implies 
    $$
    \begin{aligned}
     \sup_{t\in[0,T]}\|e^N(t)\|^p& \leq C_{T,\alpha}\big(1+\sup_{t\in[0,T]}\|u(t)\|_{\dot{H}^{\alpha}}^{3p}\big) \exp\big[C\big(\sup_{t\in[0,T]}\|u(t)\|_{\dot H^{\alpha}}^{2p}+\sup_{t\in[0,T]}\|u_N(t)\|_{\dot H^{\alpha}}^{2p}\big)\big]N^{-p\tilde{\alpha}}.
    \end{aligned}
    $$
 Taking the expectation over the set $\1_{\Omega_{L,\alpha}}$ we obtain the result stated in the lemma. 
\end{proof}

Now we  prove the convergence result for the auxiliary error.    
\begin{lemma}\label{aux_conv}
Under the assumptions on the noise and initial conditions as in Theorem~\ref{conv},  for all $N\in\N$ and $\dt\in (0,1)$ holds 
 $$
 \E\Big[\1_{\Omega_{L,\alpha}}\sup_{0\leq t\leq T}\|u_{N,\dt}(t)-u_{N}(t)\|^2\Big]\leq C_{T,L}\, \dt^2,$$ 
 where $C_{T,L}>0$.
\end{lemma}
\begin{proof}
The error $e_{N,\dt}(t):=u_{N,\dt}(t)-u_{N}(t)$ satisfies 
\begin{equation}\label{prob:error_ax}
\begin{aligned} 
\partial_t e_{N,\dt} &=A e_{N,\dt} +P_N\big(\Psi_{\dt}(u_{N,\dt})-\Psi_{0}(u_{N})\big),\\
e_{N,\dt}(0)&=0.
\end{aligned}
\end{equation}
 Considering $\bar{e}_{N,\dt}$ as a test function in the weak formulation of \eqref{prob:error_ax} and using the estimates in Lemma~\ref{Psi_loc_lip} and the moment bounds for $u_{N,\dt}$ and $u_N$ imply
\begin{equation*}
    \begin{split}
        \frac{1}{2}\frac{d}{dt}\|e_{N,\dt}(t)\|^2
       & \leq \big\langle \Psi_{\dt}(u_{N,\dt}(t))-\Psi_{\dt}(u_{N}(t)), e_{N,\dt}(t)\big\rangle_\R  + \big\langle \Psi_{\dt}(u_{N}(t))-\Psi_{0}(u_{N}(t)), e_{N,\dt}(t)\big\rangle_\R \\
         &\leq C_{T, L}\|e_{N,\dt}(t)\|^2 +\|\Psi_{\dt}(u_{N}(t))-\Psi_{0}(u_{N}(t))\|^2  \leq  C_{T, L}\big(\dt^2+\|e_{N,\dt}(t)\|^2\big).
    \end{split}
\end{equation*}
Applying Grönwall's lemma yields the result.  
\end{proof}
To estimate the time discretization error, we  need to prove a time regularity result for the auxiliary SPDE~\eqref{eq:auxeq}.
\begin{lemma}\label{time_reg}
Under the assumptions in Theorem~\ref{conv},   for all $N\in\N$ and $\beta\in \big(0,\min\{\frac{\alpha}{2},\frac{1}{2}\}\big)$ we have 
    $$
    \E\big[\1_{\Omega_{L,\alpha}}\|u_{N,\dt}(t)-u_{N,\dt}(s)\|^2\big]\leq C_{T,L,\beta}|t-s|^{2\beta},
    $$
    where $C_{T,L, \beta}>0$.
\end{lemma}
\begin{proof} 
Using~\eqref{eq:spec_mild} for the mild solution $u_{N, \dt}$ of the auxiliary SPDE~\eqref{eq:auxeq} yields
\begin{equation*}
    \begin{split}  
        \E\big[\1_{\Omega_{L,\alpha}}\|u_{N,\dt}(t)-u_{N,\dt}(s)\|^2\big]
      \! & \leq   C\Big(\|P_N(e^{tA}-e^{sA})u_0\|^2+\E\big[\1_{\Omega_{L,\alpha}}\|P_N(\eta(t)-\eta(s))\|^2\big] \Big)\\
        &\quad +C(t-s)\int_s^{t}\E\big[\1_{\Omega_{L,\alpha}}\|e^{(t-\tau) A}P_N\Psi_{\dt}(u_{N,\dt}(\tau))\|^2\big]d\tau\\
        &\quad +C_T\int_0^s\E\big[\1_{\Omega_{L,\alpha}}\big\|e^{(s-\tau)A}(e^{(t-s)A}-I)P_N\Psi_{\dt}(u_{N,\dt}(\tau)\big)\big\|^2\big]d\tau\\
        & =  I+II+III+IV.
    \end{split}
\end{equation*}
Using Lemma~\ref{lem:semigroup} and  $u_0\in \dot H^{2\beta}$, for the first term  we obtain 
\begin{equation*}
\begin{split}
   I&\leq C\|e^{sA}(e^{(t-s)A}-I)u_0\|^2=C\|e^{sA}\tilde\Delta^{-\beta}(e^{(t-s)A}-I)\tilde\Delta^{\beta}u_0\|^2\\
    &\leq C\|e^{sA}\|^2_{\mathcal{L}(L^2)}\|\tilde\Delta^{-\beta}(e^{(t-s)A}-I)\|^2_{\mathcal{L}(L^2)}\|\tilde\Delta^{\beta}u_0\|^2\leq C_T|t-s|^{2\beta}\|u_0\|^2_{\dot{H}^{2\beta}}.
\end{split}
\end{equation*}
The estimate for $II$ is proven in Lemma \ref{stoch_time_reg}.
For   $u_{N,\dt} \in \Omega_{L,\alpha}$ the third term is estimated as 
\begin{equation*}
\begin{aligned}
    III  & \leq C|t-s|\int_s^{t}\E\big[\1_{\Omega_{L,\alpha}}\|\Psi_{\dt}(u_{N,\dt}(\tau))\|^2\big]d\tau
   \\
   & \leq C|t-s|\int_s^{t}\E\big[\1_{\Omega_{L,\alpha}}(1+L^6)\big]d\tau \leq C_{T,L}|t-s|^2.
   \end{aligned}
\end{equation*}
Using Lemma~\ref{lem:semigroup}, $2\beta<1$,  and $u_{N,\dt}\in\Omega_{L,\alpha}$, for the last term we obtain 
\begin{equation*}
\begin{split}
  IV  &\leq C_T\int_0^s\E\big[\1_{\Omega_{L,\alpha}}\|\tilde \Delta^{\beta}e^{(s-\tau)A}\tilde \Delta^{-\beta}(e^{(t-s)A}-I)\Psi_{\dt}(u_{N,\dt}(\tau))\|^2\big]d\tau\\
    &\leq C_T\int_0^s\E\big[\1_{\Omega_{L,\alpha}}\|\tilde \Delta^{\beta}e^{(s-\tau)A}\|_{\mathcal{L}(L^2)}^2\|\tilde \Delta^{-\beta}(e^{(t-s)A}-I)\|_{\mathcal{L}(L^2)}^2\|\Psi_{\dt}(u_{N,\dt}(\tau))\|^2\big]d\tau\\
    &\leq C_T\int_0^s\E\big[\1_{\Omega_{L,\alpha}}\|\tilde \Delta^{\beta}e^{(s-\tau)A}\|_{\mathcal{L}(L^2)}^2\|\tilde \Delta^{-\beta}(e^{(t-s)A}-I)\|_{\mathcal{L}(L^2)}^2\|\Psi_{\dt}(u_{N,\dt}(\tau))\|^2\big]d\tau\\
    &\leq C_T\int_0^s\frac{|t-s|^{2\beta}}{|s-\tau|^{2\beta}}\E\big[\1_{\Omega_{L,\alpha}}\|\Psi_{\dt}(u_{N,\dt}(\tau))\|^2\big]d\tau \leq C_{T,L,\beta}|t-s|^{2\beta}.
\end{split}
\end{equation*}
Combining the estimates above yields the result stated in the lemma. 
\end{proof}
Now we can prove  the estimates for the time discretization error.
\begin{lemma}\label{split_conv}
Under the assumptions in Theorem~\ref{conv} it
holds, with $C_{T,L,\beta}>0$,
 $$\E\big[\1_{\Omega_{L,\alpha}}\sup_{0\leq m\leq M}\|U_{N,\dt}^m-u_{N,\dt}(t)\|^2\big]\leq C_{T,L,\beta}\dt^{2\beta}.$$
\end{lemma}

\begin{proof}
The formulas for the mild solutions $U_{N,\dt}^m$ and $u_{N,\dt}$ differ only in the nonlinear  terms. Hence we have 
\begin{equation*}
    \begin{split}
     U^{M}_{N,\dt}-u_{N,\dt}(T)
     &= \dt\sum_{m=0}^{M-1} e^{t_{M-m} A}\Big[\Psi_{\dt}(U_{N,\dt}^m)-\Psi_{\dt}(u_{N,\dt}(t_m))\Big]\\
        &\quad+\sum_{m=0}^{M-1}\int_0^{\dt} e^{t_{M-m} A}\big[\Psi_{\dt}(u_{N,\dt}(t_m))-\Psi_{\dt}(u_{N,\dt}(t_m+s))\big]ds\\
        &\quad+\sum_{m=0}^{M-1}\int_0^{\dt} \big[e^{t_{M-m} A}-e^{(t_{M-m}-s)A}\big]\Psi_{\dt}(u_{N,\dt}(t_m+s))ds.
    \end{split}
\end{equation*}
Taking the expectation yields the following three parts of the error 
\begin{equation*}
    \begin{aligned}
     &   \E\Big[\1_{\Omega_{L,\alpha}}\big\|U_{N,\dt}^{M}-u_{N,\dt}(T)\big\|^2\Big]
        \\
     &   \leq C_T\dt\sum_{m=0}^{M-1} \E\Big[\1_{\Omega_{L,\alpha}}\big\|e^{t_{M-m} A}[\Psi_{\dt}(U_{N,\dt}^m)-\Psi_{\dt}(u_{N,\dt}(t_m))]\big\|^2 \Big]\\
      & \quad  +C_T\!\sum_{m=0}^{M-1}\!\int_0^{\dt}\!\!\! \E \Big[\1_{\Omega_{L,\alpha}}\big\|e^{t_{M-m} A}[\Psi_{\dt}(u_{N,\dt}(t_m+s))-\Psi_{\dt}(u_{N,\dt}(t_m))]\big\|^2\Big]ds\\
     &  \quad  +C_T\!\sum_{m=0}^{M-1}\!\int_0^{\dt}\!\!\! \E\Big[\1_{\Omega_{L,\alpha}}\big\|[e^{(t_{M-m}-s)A}-e^{t_{M-m} A}]\Psi_{\dt}(u_{N,\dt}(t_m+s))\big\|^2\Big]ds
      \\ &  =: I+II+III.
    \end{aligned}
\end{equation*}
Using 
$
\|e^{t_{M-m} A} w\| \leq C t_{M-m}^{-1/4}\| w\|_{L^1}  $ for    $w \in L^1$, 
see e.g.~\cite{Oliver}, and  the bounded for $L^4$-norm of $U_{N, \dt}^{m}$ shown in Lemma~\ref{full_disc_moment_b}, the first term  can be estimated as  
$$
\begin{aligned}
 I &  \leq C_T \dt \sum_{m=0}^{M-1} t_{M-m}^{-\frac 12}
 \mathbb E \Big[ {\bf 1}_{\Omega_{L, \alpha} }\|
  \Psi_{\dt}(U_{N,\dt}^{m}) - \Psi_{\dt}(u_{N, \dt}(t_m))\|_{L^1}^2\Big]\\
 &\leq  C_T \dt \sum_{m=0}^{M-1} t_{M-m}^{-\frac 12}
 \mathbb E \Big[ {\bf 1}_{\Omega_{L, \alpha}}\big( \| U_{N,\dt}^{m}\|_{L^4}^2 + \| u_{N, \dt}(t_m)\|_{L^4}^2 \big) \|
  U_{N, \dt}^{m} -  u_{N, \dt}(t_m)\|^2\Big]
 \\
 & \leq  C_{T,L} \dt  \sum_{m=0}^{M-1} t_{M-m}^{-\frac 12}
 \E \Big[ \1_{\Omega_{L, \alpha}}\|
  U_{N, \dt}^{m} -  u_{N, \dt}(t_m)\|^2\Big].
\end{aligned}
$$
For part $II$, we use Lemmas~\ref{Psi_loc_lip} and~\ref{time_reg} to obtain
\begin{equation*} 
    \begin{split}
        II&\leq C_{T,L} \sum_{m=0}^{M-1}\int_0^{\dt} \E\Big[\1_{\Omega_{L,\alpha}}\|u_{N,\dt}(t_m+s)-u_{N,\dt}(t_m)\|^2\Big]ds\\
        &\leq C_{T,L} \sum_{m=0}^{M-1}\int_0^{\dt} s^{2\beta}ds\leq C_{T,L} \dt\sum_{m=0}^{M-1}\dt^{2\beta}\leq C_{T,L}T\dt^{2\beta}.
    \end{split}
\end{equation*}
For part $III$,  Lemma~\ref{lem:semigroup} implies 
\begin{equation*}
    \begin{split}
        III
        &\leq C_T\sum_{m=0}^{M-1}\int_0^{\dt} \E\Big[\1_{\Omega_{L,\alpha}}\|\tilde\Delta^{\beta}e^{(t_{M-m}-s)A}\|^2_{\mathcal{L}(L^2)}\|\tilde\Delta^{-\beta}(I-e^{sA})\|^2_{\mathcal{L}(L^2)} \|\Psi_{\dt}(u_{N,\dt}(t_m+s))\|^2\Big]ds\\
      & \leq C_T\sum_{m=0}^{M-1}\int_0^{\dt} \frac{s^{2\beta}}{(t_{M-m}-s)^{2\beta}}\E\big[\1_{\Omega_{L,\alpha}}\|\Psi_{\dt}(u_{N,\dt}(t_m+s))\|^2\big]ds.\\
        &\leq C_{T,L}\sum_{m=0}^{M-1}\int_0^{\dt} \frac{\dt^{2\beta}}{(t_{M-m}-s)^{2\beta}}ds\leq C_{T,L}\dt^{2\beta}\int_0^{T} \frac{ds}{s^{2\beta}}\leq C_{T,L,\beta}T\dt^{2\beta}.
    \end{split}
\end{equation*}
The estimates  for any $0<m<M$ and corresponding $0<t_m <T$ are performed in the same way. 
Combining the estimates above and applying  the discrete Grönwall inequality yield the result stated in the lemma. 
Note, for part $III$ we can obtain a sharper result by using the regularity of $\Psi_{\dt}(u_{N,\dt})$ and restrictions on $\beta$ are dictated by the estimates for part $II$.
\end{proof}

\begin{proof}[Proof of Theorem \ref{conv}]
 The result stated in the theorem follows by combining the estimates in Theorem~\ref{Space_convergence} and Lemmas~\ref{aux_conv} and~\ref{split_conv}.
\end{proof}
\newpage
\section{Numerical results}\label{sec_exp}

Our numerical method for solving \eqref{eq:CGLE}, defined in \eqref{eq:Full} and denoted below by \ESM, is given as  
\begin{equation*}\label{eq:FullRepeat}
U^{m+1}_{N,\dt}=e^{\dt A}P_N\Phi_{\dt}(U_{N,\dt}^m)+\sigma\int_{t_m}^{t_{m+1}}\!\!\! e^{(t_{m+1}-s)A}dW_N(s),
\end{equation*}
where $\Phi_\dt$ is specified in \eqref{eq:Phi}.

\subsection*{Notes on the implementation of \ESM}
We compute  $e^{tA}u$ in Fourier space, and calculate $\Phi_{\dt}(z)$ pointwise. For $P_N$ we use the discrete Fourier transform, yielding an aliasing error that was not considered in the error analysis. 
To evaluate the stochastic convolution, defined by \eqref{def_stoch_conv}, we have to sample from the random processes
$$a_k(t)=\int_0^t \sqrt{q_k}e^{(1+i\nu)\lambda_k(t-s)}d\beta_k^r(s)+i\int_0^t \sqrt{q_k}e^{(1+i\nu)\lambda_k(t-s)}d\beta_k^i(s),$$
for each $k\in \mathbb Z$. The random processes, $a_k$, have the same distribution as the solution of
\begin{equation}\label{eq:ak}
da_k=-(1+i\nu)\lambda_k a_k+ \sqrt{q_k} (d\beta_k^r+id\beta_k^i).
\end{equation}
Writing \eqref{eq:ak} in $\R^2$ for real and imaginary parts of $a_k$ %representation, 
we obtain a two-dimensional Ornstein-Uhlenbeck~(OU) process
$$d{\bf a}_k=\begin{pmatrix}
    -\lambda_k & \nu\lambda_k \\ -\nu\lambda_k & -\lambda_k\\
\end{pmatrix}{\bf a}_k+\sqrt{q_k}d{\boldsymbol{\beta}}_{k},$$
where ${\bf a}_k=(\Re(a_k),\Im(a_k))^T$ and $\boldsymbol{\beta}_k=(d\beta_k^r,d\beta_k^i)^T$ are i.i.d.. The following lemma gives the distribution of $a_k$. % the Fourier coefficients.
The proof can be found in \cite[Corollary 1]{vatiwutipong_alternative_2019}.
\begin{lemma}\label{Normal}
For all $a,b\in\R$, consider the two-dimensional OU process
$$d{\bf u}=-\begin{pmatrix}a & -b \\ b & a\end{pmatrix}{ \bf u}+\tilde \sigma d{\boldsymbol{\beta}},$$
with ${\bf u}_0=(0,0)^T$ and ${\boldsymbol\beta}=(\beta_1,\beta_2)^T$, where $\beta_{1}$ and $\beta_2$ are independent Brownian motions. Then ${\bf u}(t)$ is a multivariate Gaussian with mean zero and a diagonal covariance matrix with diagonal elements $\tilde \sigma^2\frac{1-e^{-2a t}}{2a}$ if $a\neq0$, otherwise the diagonal elements are $\tilde\sigma^2 t$, for $t\geq 0$.
\end{lemma}
Lemma~\ref{Normal} implies that the projected stochastic convolution can be sampled by
$$\sigma P_N\sum_{k\in\Z}\sqrt{q_k}(X_k^r+iX_k^i)\phi_k,$$
where $X_k^r$, $X_k^i$ are independent random variables with a distribution $N(0,\frac{1-e^{-2\lambda_kt}}{2\lambda_k})$. The code accompanies the arXiv version of the paper.

\subsection*{Settings for numerical simulations}
In all numerical simulation results presented here we fix the parameters $R=2^{12}$, $\sigma=2^6$ and solve~\eqref{eq:CGLE} on $[0,T]$, with $T=2^{-12}$, and with initial data 
$u_0(x)=0$ for all $x\in[0,1]$. For parameters $\mu$ and $\nu$ we examine two different sets  that, in the deterministic case, lead to different asymptotic behaviour, see~\cite{GarcaMorales2012,doering_low-dimensional_1988}.  
We consider $\mu=\nu=1$, also called the \textit{stable setting} in the deterministic case as periodic travelling waves are stable to certain perturbations~\cite{doering_low-dimensional_1988}.
These parameter values satisfy  assumptions  in Theorem~\ref{conv}.
The other case has $\mu=-3$ and $\nu=3$, and as $|\nu|>\sqrt{3}$ the assumptions required in the proof for the bound of the fourth moment for solutions $U^m_{N, \dt}$ of the fully discrete problem~\eqref{eq:Full}, which are used in the proof of the convergence results, are not satisfied. 
This case we will call the \textit{defect turbulence setting} in accordance with \cite{GarcaMorales2012}. See \cite{GarcaMorales2012}, for more information about the behaviour of this setting. 

We consider two different choices of the noise, taking different $q_k$ in Definition~\ref{qk}, and in both cases we fix $\varepsilon=5\times10^{-4}$.  
The first choice is regular noise with $r=0$, so $q_k=|k|^{-1-10^{-3}}$ for $k\neq 0$ and $q_0=1$. 
The second choice is space-time white noise ($r=-\tfrac{1}{2}$)
where $q_k=1$ for all $k$.
Before discussing the numerical convergence results, we present some sample solutions computed using \ESM with $N=2^{13}$ spatial modes and $\dt=5\cdot2^{-23}$.
In Figure \ref{fig:stab1} we plot one sample realization for the defect turbulence setting ($\mu=-3, \nu=3$)  with regular noise ($r=0$).

\begin{figure}[ht]
    \centering
   \begin{subfigure}[t]{0.24\textwidth}
        \includegraphics[width=\textwidth]{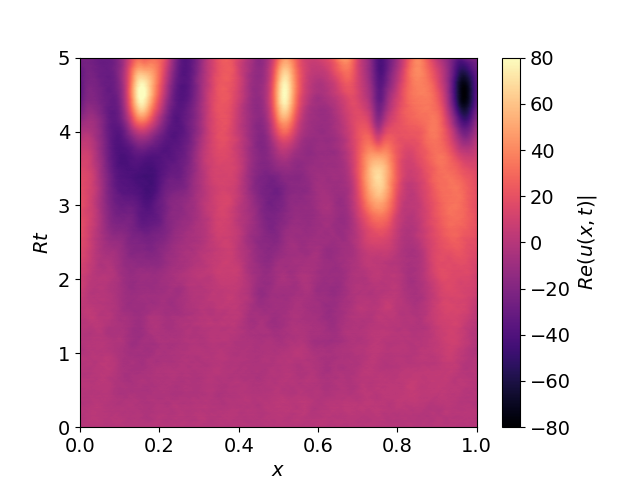} 
        \caption{}
    \end{subfigure}
    \begin{subfigure}[t]{0.24\textwidth}
        \includegraphics[width=\textwidth]{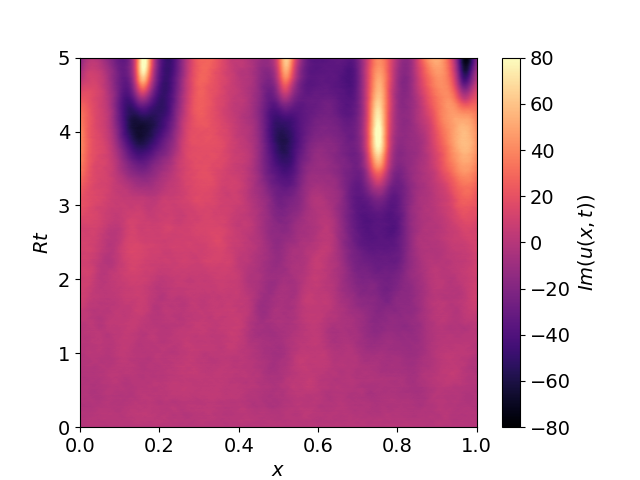} 
        \caption{}
    \end{subfigure}
    \begin{subfigure}[t]{0.24\textwidth}
        \includegraphics[width=\textwidth]{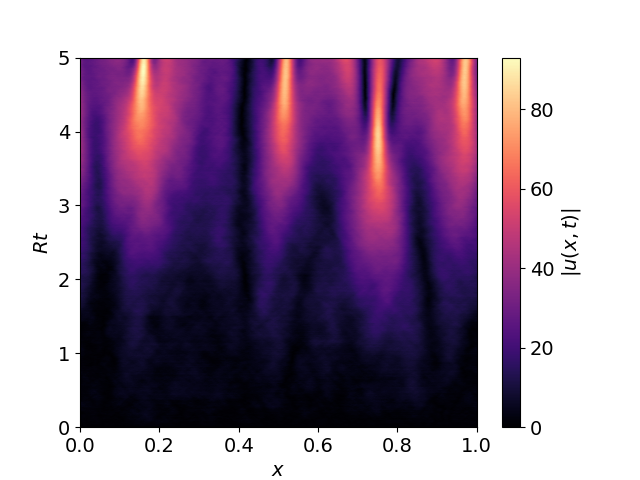} 
        \caption{}
    \end{subfigure}
\begin{subfigure}[t]{0.24\textwidth}
        \includegraphics[width=\textwidth]{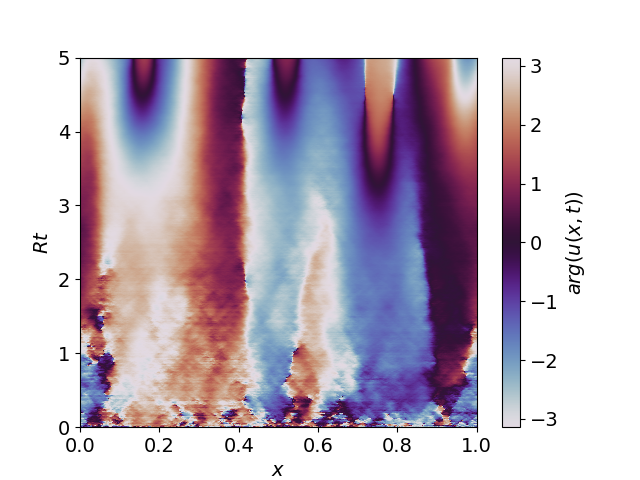} 
        \caption{}
    \end{subfigure}

     \caption{Defect turbulence setting with $\mu=-3$, $\nu=3$ and  regular noise ($r=0$). (a) $\Re (u(x,t))$, (b) $\Im (u(x,t))$, (c) $|u(x,t)|$, and (d) the phase of the solution.
     }
     \label{fig:stab1}
\end{figure}

\subsection{Space-time convergence}
To examine strong convergence in space and time numerically we estimate the root mean square error (RMSE) as 
$$\E\big(\|U^{M}_{N,\dt}-U^{M}_{2N,\frac 1 4\dt}\|^2\big)^{\frac{1}{2}}
\approx \sqrt{\frac{1}{J}\sum_{j=1}^{J}\|U^{M}_{N,\dt}(\omega_j)-U^{4M}_{2N,\frac{\dt}{4}}(\omega_j)\|^2}$$
with $J=50$. 
We ensure a parabolic scaling between the spatial and time discretizations as we refine the grids by taking $N^2\dt=1$.  

We compare our method to two other numerical methods. The first one is the exponential splitting method denoted (\ExpSM)
and the second is a tamed accelerated exponential Euler method denoted (\TAM).
The exponential splitting method (\ExpSM) is derived by splitting the SCGLE \eqref{eq:spec_CGLE} into the nonlinear, stochastic and linear term (split in this order)
\begin{equation}
\label{eq:ExpSM}
U^{m+1}_{N,\dt,\text{\ExpSM}}=e^{\dt A}(P_N\Phi_{\dt}(U^m_{N,\dt,\text{\ExpSM}}))+\sigma e^{\dt A}(W_N(t_{m+1})-W_N(t_m)).
\end{equation}
In comparison to our \ESM method the stochastic integral is not sampled exactly.

The tamed accelerated exponential Euler method (\TAM) was proposed in \cite{wang_efficient_2020} for the stochastic Allen-Cahn equation and we apply it here to the stochastic complex Ginzburg--Landau equation to get
\begin{equation*}
\label{eq:TAM}
    \begin{split}
       U^{m+1}_{N,\dt,\text{\TAM}}&=e^{\dt A}(U^{m}_{N,\dt,\text{\TAM}}) +\sigma\int_{m\dt}^{(m+1)\dt}\!\!\!\! e^{((m+1)\dt-s)A}dW(s)
      \\
      & +\left(1+\dt\|\Phi_{0}(U^m_{N,\dt,\text{\TAM}})\|\right)^{-1}\!\!\int_{m\dt}^{(m+1)\dt}\!\!\! \!\!\! e^{((m+1)\dt-s)A}ds\, \Phi_{0}(U^m_{N,\dt,\text{\TAM}}).
    \end{split}
\end{equation*}
For the Allen-Cahn equation, \cite{wang_efficient_2020} proved a strong convergence rate of order $\frac{1}{2}$ with additive space-time white noise and we are not aware of any convergence results for this method for the SCGLE~\eqref{eq:CGLE}.

We examine both stable setting, Figure~\ref{fig:stabc},  and defect turbulence setting, Figure~\ref{fig:chaosc}, with regular noise  $r=0$ in (A) and 
space-time white noise in (B). For simulations results presented in  Figure~\ref{fig:stabc}~(A),  the assumptions of Theorem~\ref{conv} are satisfied and  we observe \ESM converges in agreement with the theory.  In Figure~\ref{fig:chaosc}~(A)  we observe the same rate of convergence as in Figure~\ref{fig:stabc}~(A) even when the condition on $\nu$ is not satisfied, which was necessary for the proof of the moment bounds for solutions of the fully discrete problem \eqref{eq:Full}. In Figure~\ref{fig:stabc}~(B) and~\ref{fig:chaosc}~(B), \ESM shows a convergence rate, one would expect by extrapolating the result of Theorem~\ref{conv}. In all the convergence plots, \ExpSM has similar convergence rate as \ESM, but with larger constant. In Figures~\ref{fig:stabc}~(B) and~\ref{fig:chaosc}~(B), \TAM shows almost the same convergence behaviour as \ESM. Only in Figures~\ref{fig:stabc}~(A) and~\ref{fig:chaosc}~(A), \TAM seems to converge for larger $\dt$ values with order 1 but then approaches the same rate as \ESM for smaller values. This is likely due to the fact that  
it may be possible to prove 
time convergence of order $1$, similar to~\cite{Djurdjevac}, and that for larger $\dt$ the temporal error is dominating for \TAM, whereas for \ESM the spatial error is already dominating for these large $\dt$ values. 
\newpage
\begin{figure}[ht]
	\centering
	\begin{subfigure}[t]{0.48\textwidth}
		\includegraphics[width=\textwidth]{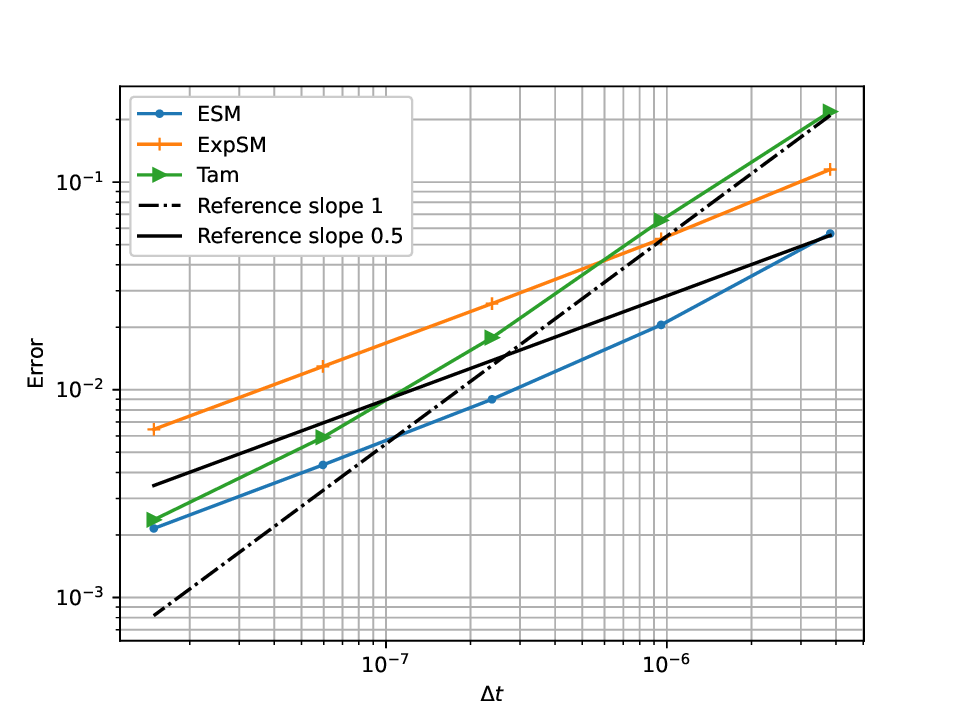}
		
		\caption{}
	\end{subfigure}
	\begin{subfigure}[t]{0.48\textwidth}
		\includegraphics[width=\textwidth]{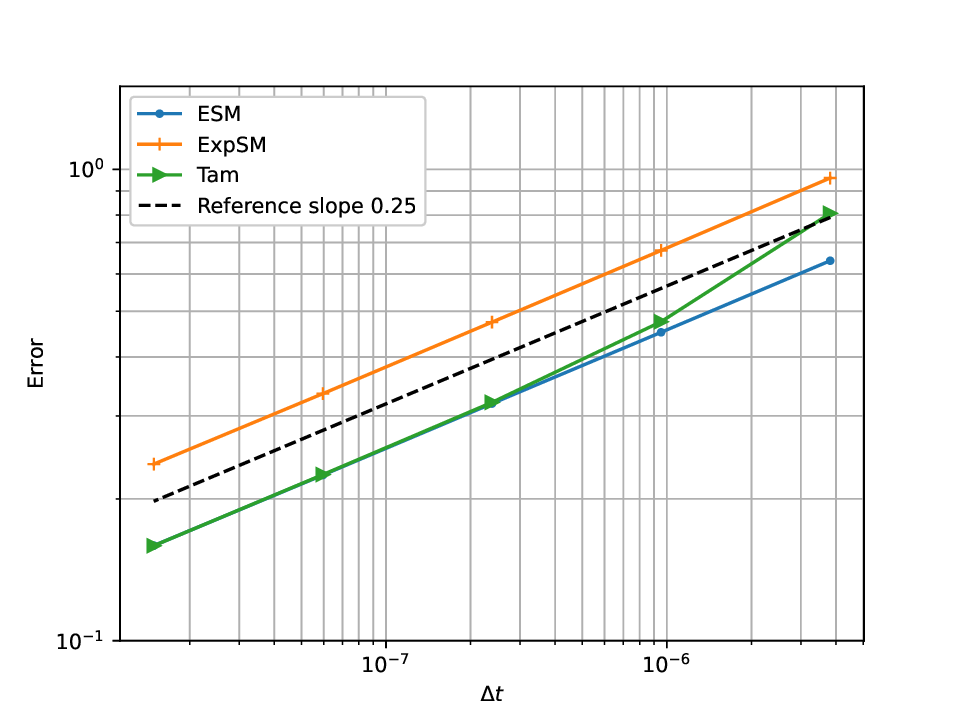} 
		\caption{}
	\end{subfigure}
	
	\caption{Convergence plot for the stable case $\mu=\nu=1$. (a) With regular noise ($r=0$)  and (b) space-time white noise ($r=-\tfrac{1}{2}$)}
	\label{fig:stabc}
\end{figure}

\begin{figure}[ht]
	\centering
	\begin{subfigure}[t]{0.48\textwidth}
		\includegraphics[width=\textwidth]{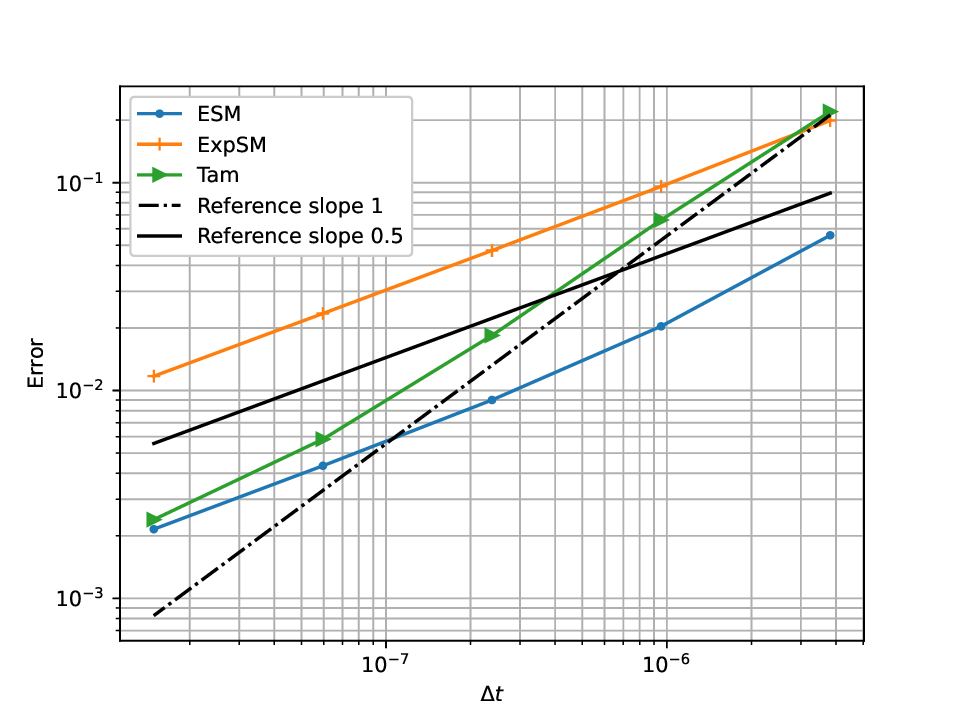}  
		\caption{}
	\end{subfigure}
	\begin{subfigure}[t]{0.48\textwidth}
		\includegraphics[width=\textwidth]{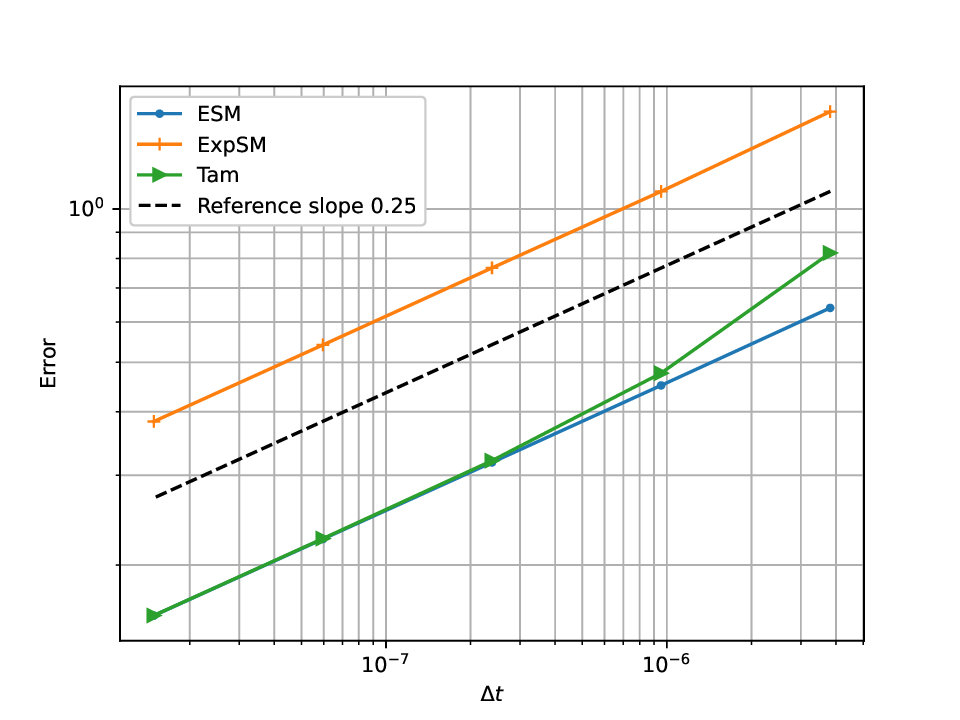} 
		
		\caption{}
	\end{subfigure}
	\caption{Convergence plot for the defect turbulence case $\mu=-3$ and $\nu=3$. (a) With regular noise ($r=0$) and (b) with space-time white noise ($r=-\tfrac{1}{2}$).}
	\label{fig:chaosc}
\end{figure}
\section*{Acknowledgments}
MJ was supported by the Swedish Research Council under the grants 2023-03930, eSSENCE; The e-Science Collaboration and the Crafoord foundation.
\printbibliography

\end{document}